\documentclass[11pt]{article}

\usepackage[margin=1in]{geometry}
\usepackage[T1]{fontenc}
\usepackage[utf8]{inputenc}
\usepackage{lmodern}
\usepackage{microtype}
\usepackage{amsmath,amssymb,amsthm,mathtools}
\usepackage{enumitem}
\usepackage{booktabs}
\usepackage[hypertexnames=false]{hyperref}

\hypersetup{
  colorlinks=true,
  linkcolor=blue,
  citecolor=blue,
  urlcolor=blue
}

\allowdisplaybreaks
\numberwithin{equation}{section}

\newtheorem{theorem}{Theorem}[section]
\newtheorem{proposition}[theorem]{Proposition}
\newtheorem{lemma}[theorem]{Lemma}
\newtheorem{corollary}[theorem]{Corollary}
\theoremstyle{remark}
\newtheorem{remark}[theorem]{Remark}

\newcommand{\R}{\mathbb R}
\newcommand{\Prob}{\mathbb P}
\newcommand{\E}{\mathbb E}
\newcommand{\1}{\mathbf 1}
\newcommand{\supp}{\operatorname{supp}}
\newcommand{\dd}{\,d}
\newcommand{\email}[1]{\href{mailto:#1}{#1}}

\title{Singleton Optimality in Standard Quadratic Programs with the GOE}
\author{Xin Chen\\
H. Milton Stewart School of Industrial and Systems Engineering\\
Georgia Institute of Technology, Atlanta, GA, USA\\
ORCID: \url{https://orcid.org/0000-0002-5168-4823}\\
\email{xin.chen@isye.gatech.edu}}
\date{\today}

\begin{document}
\maketitle
\begin{abstract}
We study the standard quadratic optimization problem over the simplex when the objective matrix is drawn from the Gaussian Orthogonal Ensemble (GOE). Let \(\kappa_n\) denote the support size of the almost surely unique global optimizer. We prove
\[
\Prob(\kappa_n>1)\sim 2\sqrt{2\pi}\,\frac{\sqrt{\log n}}{n}.
\]
The proof combines an exact two-coordinate condition for edge improvement with a product formula obtained by conditioning on the diagonal order statistics. Boundary-layer estimates identify the leading contribution and show that supports of size at least three are negligible. Consequently, the minimum-diagonal vertex is globally optimal with probability tending to one, with an explicit first-order correction.
\end{abstract}


\section{Introduction}

The standard quadratic optimization problem (StQP) over the probability simplex is
\begin{equation}
\min\{x^\top Qx:x\in\Delta^{n-1}\},\qquad
\Delta^{n-1}:=\{x\in\R_+^n:\1^\top x=1\}.
\label{eq:1.1}
\end{equation}
Here \(\R_+:=[0,\infty)\), \(\R_+^n\) is the nonnegative orthant, and \(\1\) is the all-ones vector. This problem is a basic nonconvex quadratic optimization model with connections to copositive programming, continuous formulations of discrete optimization problems, portfolio selection, and sparse optimization; see Bomze \cite{bomze1998} and Bomze--de Klerk \cite{bomze-deklerk}. Recent work has also developed tight relaxations, instance generation, uncertainty models, and exactness results for doubly nonnegative relaxations of random standard quadratic programs \cite{bomze-peng-qiu-yildirim,bomze-vicente,chen-dnn}.

This paper considers the Gaussian Orthogonal Ensemble (GOE) model in the following normalization. The matrix \(Q\) is symmetric and
\begin{equation}
Q_{ii}=Z_i\sim N(0,1),\qquad
Q_{ij}=Q_{ji}\sim N(0,1/2),\quad i<j,
\label{eq:1.2}
\end{equation}
with all upper-triangular entries independent. Here and throughout, \(Y\sim\mathcal D\) means that the random variable \(Y\) has distribution \(\mathcal D\). This is one conventional unnormalized GOE scaling: the distribution is orthogonally invariant, and every off-diagonal entry has variance \(1/2\). Throughout, \(\Phi\) and \(\phi\) denote the standard normal distribution function and density. For a vector \(x\), write \(\supp(x):=\{i:x_i>0\}\). For each \(n\), let \(x^\star\) be the almost surely unique global optimizer of \eqref{eq:1.1}; the dependence of \(x^\star\) on \(n\) is suppressed. Let \(\kappa_n\) be the cardinality of its support:
\[
\kappa_n:=|\supp(x^\star)|.
\]
We are interested in analyzing the probability \(\Prob(\kappa_n=1)\) that the optimizer is a vertex of the simplex. All asymptotic notation is for \(n\to\infty\). In particular, \(a_n\sim b_n\) means \(a_n/b_n\to1\); this asymptotic use of \(\sim\) is distinct from the distributional notation \(Y\sim\mathcal D\). The symbols \(O(\cdot)\) and \(o(\cdot)\) have their usual meanings.

\subsection{Literature review and contribution}

The sparsity theory for random standard quadratic programs was initiated by Chen, Peng, and Zhang \cite{chen-peng-zhang}, who showed that random instances typically admit sparse optimal solutions. Chen and Peng \cite{chen-peng} sharpened these bounds and proved, in particular, that in the GOE setting \(\Prob(\kappa_n\le2)\to1\). Their result allows support size two as the remaining high-probability possibility, but does not determine how often that case occurs. One consequence of the present paper is that \(\Prob(\kappa_n=2)\) is itself small, of order \(\sqrt{\log n}/n\). Chen and Pittel \cite{chen-pittel} developed further probabilistic refinements and generalizations. These works rely primarily on necessary optimality conditions and union bounds. In particular, a one-row union bound replaces the full KKT system by a necessary inequality involving one active row at a time, and then sums the resulting event over possible supports. Such arguments identify the sparsity phenomenon but do not determine the exact leading constant for \(\Prob(\kappa_n>1)\).

The GOE and related Wigner ensembles are classical objects in random matrix theory; see Mehta \cite{mehta} and Tao \cite{tao}. The analysis here is entrywise rather than spectral. The relevant random quantities are the smallest diagonal order statistics and the off-diagonal entries associated with pairs of coordinates. Standard facts about order statistics and asymptotic analysis are used throughout; see David--Nagaraja \cite{david-nagaraja}, Feller \cite{feller}, and de Bruijn \cite{debruijn}.

We replace the one-row argument by an exact two-point calculation. Let
\[
Z_{(1)}\le Z_{(2)}\le\cdots\le Z_{(n)}
\]
be the ordered diagonal entries of \(Q\), where \([n]:=\{1,\ldots,n\}\). Let \(\pi\) be the almost surely unique random permutation of \([n]\) satisfying \(Q_{\pi(r)\pi(r)}=Z_{(r)}\). For \(i<j\), relabel the off-diagonal variables according to the same ordering by
\[
X_{ij}:=Q_{\pi(i)\pi(j)}.
\]
This convention is used throughout the ordered-coordinate analysis. Conditional on the diagonal variables and their ordering, the variables \(X_{ij}\), \(i<j\), are independent and each has distribution \(N(0,1/2)\).
Here \(e_r\) denotes the \(r\)-th standard basis vector of \(\R^n\). The line segment joining the two vertices \(e_{\pi(i)}\) and \(e_{\pi(j)}\) contains a point with objective value smaller than every vertex value if and only if
\[
X_{ij}<Z_{(1)}-\sqrt{(Z_{(i)}-Z_{(1)})(Z_{(j)}-Z_{(1)})}.
\]
Given the diagonal order statistics, these two-coordinate events are independent over pairs.

For every \(j\ge2\), the event that the segment for the pair \((1,j)\) goes below the best vertex value has the same conditional probability
\[
q_n:=\Phi(\sqrt2 Z_{(1)}).
\]
The probability that there exists at least one pair \((1,j)\), \(j\ge2\), whose segment beats the best vertex value is therefore
\[
S_n:=\E\left[1-(1-q_n)^{n-1}\right].
\]
After conditioning on the diagonal order statistics, different two-coordinate events use different off-diagonal entries. Therefore the probability that none of these segments beats the best vertex is the product of the corresponding conditional no-improvement probabilities. Separating the factors involving the smallest diagonal entry leaves two further contributions. The term \(A_n\) comes from pairs \((2,j)\) with \(j\ge3\), where index 2 refers to the second smallest diagonal entry. The term \(B_n\) comes from pairs \((i,j)\) with \(3\le i<j\le n\). The main result is
\[
\Prob(\kappa_n>1)=S_n+A_n+B_n+o(B_n),
\]
where
\[
A_n\sim3\sqrt{\frac\pi2}\frac{1}{n\sqrt{\log n}},\qquad
B_n\sim9\sqrt{\frac\pi2}\frac{1}{n(\log n)^{3/2}},
\]
and
\[
S_n\sim 2\sqrt{2\pi}\frac{\sqrt{\log n}}{n}.
\]
Hence
\[
\Prob(\kappa_n>1)\sim2\sqrt{2\pi}\frac{\sqrt{\log n}}{n},\qquad
\Prob(\kappa_n=1)=1-\Theta\left(\frac{\sqrt{\log n}}{n}\right).
\]

A simple calculation explains the scale. The minimum diagonal value is near
\(-\sqrt{2\log n}\). For an edge incident to this vertex, the relevant
off-diagonal tail is therefore of size
\(\Phi(\sqrt2 Z_{(1)})\), whose expectation is of order
\(\sqrt{\log n}/n^2\). Multiplying by the \(n-1\) edges incident to the
minimum diagonal vertex gives the scale \(\sqrt{\log n}/n\). The main proof
makes this argument exact, identifies the leading constant, and shows that
all other two-coordinate and higher-support contributions are smaller.

This expansion has a simple algorithmic implication. For independent GOE data, the minimum-diagonal vertex is globally optimal with probability tending to one. Thus such instances are usually solved by inspecting the diagonal and should not be regarded as generic hard benchmarks for nonconvex StQP algorithms.

The proof proceeds by combining the two-coordinate structure of simplex edges with the continuous extremes of the GOE diagonal. Section 2 gives finite-sample Gaussian tail estimates. Section 3 records deterministic optimality facts, and Section 4 derives the exact two-coordinate condition. Section 5 decomposes the conditional no-improvement product into the three terms \(S_n\), \(A_n\), and \(B_n\). Sections 6--8 estimate these terms through boundary-layer integrations, Section 9 bounds supports of size at least three, and Section 10 combines the estimates. Section 11 discusses extensions, and Appendix \ref{app:one-row} records one-row KKT relaxation probabilities indexed by the diagonal order statistic for comparison with earlier bounds.

For reference, the main notation is as follows. The ordered diagonal entries are \(Z_{(r)}\), the corresponding uniform order statistics are \(U_{(r)}=\Phi(Z_{(r)})\), and \(X_{ij}\) denotes the off-diagonal entry after relabeling by the diagonal ordering. The two-coordinate threshold is \(\tau_{ij}\), the associated event is \(F_{ij}\), and its conditional probability is \(p_{ij}\). The quantities \(S_n,A_n,B_n\) are the three terms in the product decomposition. In the integral estimates, \(p(u,v,w)\), \(G(u,v)\), and \(H(u,v)\) denote the pair probability, its one-sample average, and its two-sample average, respectively. The tail and logarithmic scales \(s_n(x)\), \(L_n(x)\), and \(\ell_n(x)\) are defined in Section \ref{sec:toolbox}.

Throughout the paper, the letters \(C\) and \(c\) denote positive finite constants whose values may change from line to line. Subscripts are used only when a constant is fixed within a local argument, for instance when it depends on a parameter such as \(\eta\).

\section{Finite-sample Gaussian tail estimates}\label{sec:toolbox}

This section records Gaussian tail estimates, based on Mills' bounds, used repeatedly in the proofs. We use
\[
\bar\Phi(x):=1-\Phi(x),
\]
and write \(\Gamma(\cdot)\) for the Euler Gamma function.

\begin{lemma}[Mills bounds and tail transfer]\label{lem:tail-transfer}
For every \(x>0\),
\begin{equation}\label{eq:mills}
\frac{x}{1+x^2}\phi(x)\le \bar\Phi(x)\le \frac{\phi(x)}{x}.
\end{equation}
Consequently, for every fixed \(a\ge1\), as \(u\downarrow0\),
\begin{equation}\label{eq:tail-transfer}
\Phi(a\Phi^{-1}(u))
=\frac{(4\pi)^{(a^2-1)/2}}{a}\,
 u^{a^2}\bigl(\log(1/u)\bigr)^{(a^2-1)/2}(1+o(1)).
\end{equation}
Moreover, for every fixed \(a\ge1\), there is a constant \(C_a<\infty\) such that
\begin{equation}\label{eq:tail-transfer-upper}
\Phi(a\Phi^{-1}(u))
\le C_a u^{a^2}\bigl(\log(1/u)\bigr)^{(a^2-1)/2},
\qquad 0<u\le1/2.
\end{equation}
\end{lemma}

\begin{proof}
The inequalities in \eqref{eq:mills} are the standard Mills bounds \cite{feller,abramowitz-stegun}. Set \(s=-\Phi^{-1}(u)\). Then \(s\to\infty\) and \(u=\Phi(-s)\). Mills' bounds imply
\begin{equation}\label{eq:mills-first-asymp}
u=\frac{\phi(s)}{s}(1+O(s^{-2})),
\end{equation}
and, with the same error order,
\begin{equation}\label{eq:mills-second-asymp}
\Phi(a\Phi^{-1}(u))=\Phi(-as)=\frac{\phi(as)}{as}(1+O(s^{-2})).
\end{equation}
Using \eqref{eq:mills-first-asymp} to express the exponential factor in \eqref{eq:mills-second-asymp},
\begin{align*}
\frac{\Phi(-as)}{u^{a^2}}
&=\frac{\phi(as)}{as}\left(\frac{s}{\phi(s)}\right)^{a^2}(1+O(s^{-2})) \\
&=\frac{(2\pi)^{(a^2-1)/2}}{a}\,s^{a^2-1}(1+O(s^{-2})).
\end{align*}
Taking logarithms in \eqref{eq:mills-first-asymp} gives
\[
\log(1/u)=\frac{s^2}{2}+\log s+\frac12\log(2\pi)+O(s^{-2}),
\]
so \(s^2\sim2\log(1/u)\). Substitution gives \eqref{eq:tail-transfer}. The upper bound \eqref{eq:tail-transfer-upper} follows from the one-sided estimates in \eqref{eq:mills}; increasing \(C_a\) covers the compact range of \(u\) bounded away from zero.
\end{proof}

\begin{corollary}[GOE off-diagonal tail transfer]\label{cor:goe-tail-transfer}
Define
\[
\Psi(u):=\Phi(\sqrt2\,\Phi^{-1}(u)),\qquad 0<u<1.
\]
Then, as \(u\downarrow0\),
\begin{equation}\label{eq:goe-tail-transfer}
\Psi(u)=\sqrt{2\pi}\,u^2\sqrt{\log(1/u)}\,(1+o(1)).
\end{equation}
Moreover, there is a constant \(C<\infty\) such that
\begin{equation}\label{eq:goe-tail-transfer-upper}
\Psi(u)\le C u^2\sqrt{\log(1/u)},\qquad 0<u\le1/2.
\end{equation}
\end{corollary}

\begin{proof}
Apply Lemma \ref{lem:tail-transfer} with \(a=\sqrt2\).
\end{proof}

\begin{remark}[Quantile sandwich]\label{rem:quantile-sandwich}
For fixed \(a\ge1\), call a pair of inequalities of the form
\[
c_a u^{a^2}\bigl(\log(1/u)\bigr)^{(a^2-1)/2}
\le
\Phi(a\Phi^{-1}(u))
\le
C_a u^{a^2}\bigl(\log(1/u)\bigr)^{(a^2-1)/2},
\qquad 0<u\le u_0,
\]
a quantile sandwich. Here \(u_0\in(0,1/2]\) and \(0<c_a<C_a<\infty\) may depend on \(a\). Lemma \ref{lem:tail-transfer} gives such a sandwich for every fixed \(a\ge1\).
\end{remark}

\begin{lemma}[Beta-log integral]\label{lem:beta-log}
Let \(r>0\) and \(\gamma\ge0\) be fixed. Then
\begin{equation}\label{eq:beta-log-asymp}
 n\int_0^1 u^r(1-u)^{n-1}\bigl(\log(1/u)\bigr)^\gamma\,\dd u
=\Gamma(r+1)n^{-r}(\log n)^\gamma(1+o(1)).
\end{equation}
Moreover, for all \(n\ge2\),
\begin{equation}\label{eq:beta-log-upper}
 n\int_0^1 u^r(1-u)^{n-1}\bigl(\log(1/u)\bigr)^\gamma\,\dd u
\le C_{r,\gamma}n^{-r}(\log n)^\gamma,
\end{equation}
where \(C_{r,\gamma}<\infty\) depends only on \(r\) and \(\gamma\).
\end{lemma}

\begin{proof}
Set \(u=t/n\). The left side of \eqref{eq:beta-log-asymp} equals
\[
 n^{-r}\int_0^n t^r\left(1-\frac{t}{n}\right)^{n-1}
 \bigl(\log(n/t)\bigr)^\gamma\,\dd t.
\]
After division by \(n^{-r}(\log n)^\gamma\), the integrand converges pointwise to \(t^re^{-t}\). For \(0<t\le n/2\),
\[
\left(1-\frac{t}{n}\right)^{n-1}\le e^{-t/2},
\]
and \((\log(n/t)/\log n)^\gamma\) is bounded by a constant multiple of \(1+|\log t|^\gamma\). Thus the part \(0<t\le n/2\) is dominated by an integrable function. For \(n/2<t<n\), the integral is exponentially small because \((1-t/n)^{n-1}\le2^{-(n-1)}\) and the remaining factors grow at most polynomially in \(n\). Dominated convergence gives \eqref{eq:beta-log-asymp}; the same estimates give \eqref{eq:beta-log-upper}.
\end{proof}

Throughout the tail estimates, for \(0<x<n\) set
\[
s_n(x):=-\Phi^{-1}(x/n),\qquad
L_n(x):=\log(n/x),\qquad
\ell_n(x):=1+\log_+(n/x),
\]
where \(\log_+(z):=\max\{0,\log z\}\). The scale \(L_n(x)\) is the exact logarithmic scale, while \(\ell_n(x)\) is a positive envelope used in global upper bounds.

\begin{lemma}[Quantile-scale identities]\label{lem:quantile-scale}
With the notation above, uniformly for \(x\) in every compact subinterval of \((0,\infty)\),
\begin{equation}\label{eq:sn-log}
s_n(x)^2\sim 2L_n(x)\sim 2\log n,
\end{equation}
and, uniformly for \(0<x\le\sqrt n\),
\begin{equation}\label{eq:sn-log-refined}
\log(n/x)=\frac{s_n(x)^2}{2}+O(\log s_n(x)).
\end{equation}
Furthermore, uniformly for \(x\) in every compact subinterval of \((0,\infty)\),
\begin{equation}\label{eq:sn-exp-id}
\frac{n^2e^{-s_n(x)^2}}{s_n(x)^2}=2\pi x^2(1+o(1)).
\end{equation}
In addition, there is a universal constant \(C<\infty\) such that, whenever \(0<x\le\sqrt n\),
\begin{equation}\label{eq:sn-exp-upper}
e^{-s_n(x)^2}\le C\frac{x^2s_n(x)^2}{n^2}.
\end{equation}
\end{lemma}

\begin{proof}
Write \(s=s_n(x)\), so that \(x/n=\Phi(-s)\). Mills' bounds give
\[
\frac{s}{1+s^2}\phi(s)\le \frac{x}{n}\le\frac{\phi(s)}{s}.
\]
Taking logarithms yields
\[
\frac{s^2}{2}=\log n-\log x+O(\log s),
\]
uniformly for \(0<x\le\sqrt n\). This proves \eqref{eq:sn-log-refined}, and \eqref{eq:sn-log} follows by restricting \(x\) to compact subsets of \((0,\infty)\). The sharper relation
\[
\frac{x}{n}=\frac{\phi(s)}{s}(1+o(1))
\]
then gives
\[
e^{-s^2}=2\pi\frac{x^2s^2}{n^2}(1+o(1)),
\]
which is \eqref{eq:sn-exp-id}. Finally, the lower Mills bound gives
\[
\phi(s)\le C\frac{x}{n}s,
\]
for \(0<x\le\sqrt n\) and all sufficiently large \(n\), which implies \eqref{eq:sn-exp-upper}; increasing \(C\) covers the remaining finite set of \(n\).
\end{proof}

The next elementary estimates will be used later to bound outer parts of the order-statistic integrals.

\begin{lemma}[Outer integral estimates]\label{lem:outer-integrals}
Let \(P(x)\) be a polynomial with nonnegative coefficients, and let \(c,c_1,\kappa,\rho>0\) and \(\gamma\ge0\). Then there is \(c_2>0\) such that, as \(n\to\infty\),
\begin{align}
\int_0^{\sqrt n}P(x)\ell_n(x)^\gamma e^{-\kappa s_n(x)^2}e^{-cx}\,\dd x
&=O\left(n^{-2\kappa}(\log n)^{\gamma+\kappa}\right),
\label{eq:outer-int-1}\\
\int_0^{\sqrt n}P(x)\ell_n(x)^\gamma e^{-c_1s_n(x)}e^{-cx}\,\dd x
&=O\left((\log n)^\gamma e^{-c_2\sqrt{\log n}}\right),
\label{eq:outer-int-2}\\
\int_0^{\sqrt n}P(x)\ell_n(x)^\gamma \left(\frac{x}{n}\right)^\rho e^{-cx}\,\dd x
&=O\left(n^{-\rho}(\log n)^\gamma\right).
\label{eq:outer-int-3}
\end{align}
\end{lemma}

\begin{proof}
The estimates are crude tail bounds used later when one order-statistic scale is separated from another. For \eqref{eq:outer-int-1}, use \eqref{eq:sn-exp-upper}. Since
\[
e^{-s_n(x)^2}\le C\frac{x^2s_n(x)^2}{n^2}
\]
uniformly on \(0<x\le\sqrt n\), raising both sides to the power \(\kappa\) gives
\[
e^{-\kappa s_n(x)^2}
\le C_\kappa\left(\frac{x}{n}\right)^{2\kappa}s_n(x)^{2\kappa}
\le C_\kappa\left(\frac{x}{n}\right)^{2\kappa}\ell_n(x)^\kappa,
\]
because \(s_n(x)^2\le C\ell_n(x)\) on \(0<x\le\sqrt n\), by Lemma \ref{lem:quantile-scale}. Hence the left-hand side of \eqref{eq:outer-int-1} is bounded by
\[
C n^{-2\kappa}\int_0^{\sqrt n}P(x)x^{2\kappa}\ell_n(x)^{\gamma+\kappa}e^{-cx}\,\dd x.
\]
Since \(\ell_n(x)\le 1+\log n+|\log x|\), the last integral is
\(O((\log n)^{\gamma+\kappa})\). This proves \eqref{eq:outer-int-1}. The same argument, without the quantile estimate, gives \eqref{eq:outer-int-3}.
For \eqref{eq:outer-int-2}, the bound \(s_n(x)\ge c'\sqrt{\log n}\) on \(0<x\le\sqrt n\) gives the factor \(e^{-c_2\sqrt{\log n}}\); the remaining integral has only polynomial-logarithmic growth and is absorbed into \((\log n)^\gamma\) after decreasing \(c_2\) if necessary.
\end{proof}

To integrate the conditional pair probabilities later in the proof, we need to translate small multiplicative changes in the uniform order-statistic scale into gaps between extreme normal quantiles. The following lemma records the three regimes used in Sections \ref{sec:A} and \ref{sec:B}. The name ``boundary layer'' refers to the first regime: in the applications below, \(s^2\) is of order \(\log n\), so the relative gap \(y/x-1=\beta/s^2\) is of order \(1/\log n\).

\begin{lemma}[Lower-tail quantile-gap estimates]\label{lem:quantile-gap}
Let
\[
0<x\le y<n,
\qquad u=\frac{x}{n},
\qquad v=\frac{y}{n},
\]
and define
\[
s:=-\Phi^{-1}(u),\qquad t:=-\Phi^{-1}(v),\qquad h:=s-t.
\]
\begin{enumerate}[label=(\roman*)]
\item \emph{Boundary scaling.} If
\[
y=x\left(1+\frac{\beta}{s^2}\right),
\]
where \(x\) ranges over a fixed compact subset of \((0,\infty)\) and \(\beta\) ranges over a fixed compact subset of \([0,\infty)\), then, uniformly as \(n\to\infty\),
\begin{equation}\label{eq:gap-boundary}
s^3h\to\beta.
\end{equation}

\item \emph{Uniform near-region bound.} Fix \(\eta\in(0,1]\). There are constants \(c_\eta,C_\eta>0\) such that the following holds for all sufficiently large \(n\). If
\[
0<x\le\sqrt n,
\qquad
0\le\beta\le\eta s^2,
\qquad
y=x\left(1+\frac{\beta}{s^2}\right),
\]
then
\begin{equation}\label{eq:gap-near-bound}
c_\eta\frac{\beta}{s^3}\le h\le C_\eta\frac{\beta}{s^3}.
\end{equation}

\item \emph{Far-ratio bound.} Fix \(\eta\in(0,1]\). There is a constant \(c_0>1/2\), for instance \(c_0=3/5\), such that the following holds for all sufficiently large \(n\), with the lower bound on \(n\) allowed to depend on \(\eta\). If
\[
0<x\le\sqrt n,
\qquad (1+\eta)x\le y<n,
\]
then
\begin{equation}\label{eq:gap-far-bound}
sh\ge c_0\log(y/x).
\end{equation}
\end{enumerate}
\end{lemma}

Part \textup{(i)} gives the pointwise constant in the boundary-layer limits. Part \textup{(ii)} is the uniform near-region estimate used for domination after the change of variables \(y=x(1+\beta/s^2)\). Part \textup{(iii)} controls the far-ratio region \(y/x\ge1+\eta\), where the quantile gap is large enough to produce additional exponential decay. The explicit lower bound \(c_0>1/2\) is needed later to make the far-ratio \(G_-\) decay stronger than the threshold \(1/\sqrt2\) used in the endpoint estimates.

\begin{proof}
We use
\begin{equation}\label{eq:gap-basic-identity}
v-u=\Phi(-t)-\Phi(-s)=\int_t^s\phi(r)\,\dd r.
\end{equation}
When \(0<t<s\), this implies
\begin{equation}\label{eq:gap-density-bracket}
\phi(s)h\le v-u\le \phi(t)h.
\end{equation}

For part \textup{(i)}, \(v-u=u\beta/s^2\). Mills' bounds give
\begin{equation}\label{eq:u-over-phis}
\frac{u}{\phi(s)}=\frac1s(1+o(1)).
\end{equation}
The lower bound in \eqref{eq:gap-density-bracket} gives
\[
h\le \frac{u\beta}{s^2\phi(s)}=O(s^{-3}),
\]
uniformly for bounded \(\beta\). Hence \(t/s\to1\), and for every \(r\in[t,s]\),
\[
\frac{\phi(r)}{\phi(s)}
=\exp\left(\frac{s^2-r^2}{2}\right)=1+o(1),
\]
because \(s(s-r)\le sh=O(s^{-2})\). By the mean-value theorem applied to \eqref{eq:gap-basic-identity},
\[
v-u=\phi(\xi)h
\]
for some \(\xi\in[t,s]\). Combining this identity with \eqref{eq:u-over-phis} gives
\[
h=\frac{u\beta/s^2}{\phi(\xi)}
=\frac{\beta+o(1)}{s^3},
\]
which proves \eqref{eq:gap-boundary}.

For part \textup{(ii)}, we have
\[
\frac vu=1+\frac{\beta}{s^2}\le1+\eta,
\qquad
v-u=\frac{u\beta}{s^2}.
\]
We first compare \(\phi(t)\) and \(\phi(s)\). From the lower Mills bound for \(v=\Phi(-t)\) and the upper Mills bound for \(u=\Phi(-s)\),
\begin{equation}\label{eq:near-ratio-mills}
\frac vu
\ge
\frac{t}{1+t^2}\,s\,
\exp\left(\frac{s^2-t^2}{2}\right).
\end{equation}
Since \(v/u\le1+\eta\), the case \(t<s/2\) is impossible for all sufficiently large \(n\), because the exponential term in \eqref{eq:near-ratio-mills} would diverge. Thus \(t\ge s/2\). The bound \eqref{eq:near-ratio-mills} then implies \(s^2-t^2\le C_\eta\), and consequently
\[
\phi(t)\le C_\eta\phi(s).
\]
Using \eqref{eq:gap-density-bracket} and \(v-u=u\beta/s^2\),
\[
\frac{u\beta}{s^2\phi(t)}\le h\le \frac{u\beta}{s^2\phi(s)}.
\]
Mills' bounds also imply constants \(c,C>0\), independent of \(x\), such that
\[
\frac{c}{s}\le \frac{u}{\phi(s)}\le \frac{C}{s}.
\]
Together with \(\phi(t)\le C_\eta\phi(s)\), this proves \eqref{eq:gap-near-bound}.

For part \textup{(iii)}, put \(a=\log(y/x)=\log(v/u)\). We prove the claim with \(c_0=3/5\). Since \(y\ge(1+\eta)x\), we have \(a\ge a_\eta:=\log(1+\eta)>0\).

First suppose that \(t\ge s/2\). Then \(t\to\infty\) uniformly in the stated range, and Mills' bounds give
\begin{equation}\label{eq:far-ratio-log-bound}
a=\frac{s^2-t^2}{2}+\log\frac{s}{t}+o(1),
\end{equation}
uniformly. Since \(h=s-t\),
\[
\frac{s^2-t^2}{2}=sh-\frac{h^2}{2}\le sh.
\]
Also \(0\le h/s\le1/2\), and therefore
\[
\log\frac{s}{t}=\log\frac1{1-h/s}\le\frac{2h}{s}\le \frac16 sh
\]
for all sufficiently large \(n\), because \(0<x\le\sqrt n\) implies \(s\ge c\sqrt{\log n}\to\infty\). Finally, because \(a\ge a_\eta\), the \(o(1)\) term in \eqref{eq:far-ratio-log-bound} is at most \(a/6\) for all sufficiently large \(n\). Hence
\[
\frac56a\le \frac76 sh,
\]
so \(sh\ge(5/7)a\). In particular \(sh\ge(3/5)a\).

Next suppose that \(t<s/2\). This also covers the case where \(t\) is small or negative, which occurs when \(y/n\) is not in the extreme lower tail. Then \(h=s-t>s/2\), so
\[
sh>\frac{s^2}{2}.
\]
On the other hand, \(y<n\) gives
\[
a=\log(y/x)\le \log(n/x).
\]
By \eqref{eq:sn-log-refined}, uniformly over \(0<x\le\sqrt n\),
\[
\log(n/x)=\frac{s^2}{2}+O(\log s)\le \frac56s^2
\]
for all sufficiently large \(n\). Consequently \(a\le (5/3)sh\), and hence \(sh\ge(3/5)a\) in this case as well. This proves \eqref{eq:gap-far-bound}.
\end{proof}

\begin{theorem}[Normal variable versus the minimum]\label{thm:normal-minimum}
Let \(M_n\) denote the minimum of \(n\) independent \(N(0,1)\) random variables, and let \(Y_a\sim N(0,a^{-2})\) be independent of \(M_n\), where \(a\ge1\) is fixed. Then
\begin{equation}\label{eq:normal-minimum-asymp}
\Prob(Y_a\le M_n)
\sim \frac{(4\pi)^{(a^2-1)/2}}{a}\,\Gamma(a^2+1)
 n^{-a^2}(\log n)^{(a^2-1)/2}.
\end{equation}
In addition, there is a constant \(C_a<\infty\) such that
\begin{equation}\label{eq:normal-minimum-upper}
\Prob(Y_a\le M_n)
\le C_a n^{-a^2}(\log n)^{(a^2-1)/2},
\qquad n\ge2.
\end{equation}
\end{theorem}

\begin{proof}
Since \(U:=\Phi(M_n)\) has density \(n(1-u)^{n-1}\) on \((0,1)\),
\[
\Prob(Y_a\le M_n)=n\int_0^1\Phi(a\Phi^{-1}(u))(1-u)^{n-1}\,\dd u.
\]
The integral over \([1/2,1]\) is exponentially small. On \((0,1/2]\), Lemma \ref{lem:tail-transfer} gives
\[
\Phi(a\Phi^{-1}(u))
=\frac{(4\pi)^{(a^2-1)/2}}{a}\,
 u^{a^2}\bigl(\log(1/u)\bigr)^{(a^2-1)/2}(1+o(1))
\]
near zero, and the upper bound \eqref{eq:tail-transfer-upper} supplies domination. Lemma \ref{lem:beta-log}, with \(r=a^2\) and \(\gamma=(a^2-1)/2\), gives \eqref{eq:normal-minimum-asymp}. The same two lemmas give \eqref{eq:normal-minimum-upper}.
\end{proof}

\section{Deterministic optimality}\label{sec:deterministic}

For a finite set \(K\), \(|K|\) denotes its cardinality, and \(\1_K\) denotes the all-ones vector of dimension \(|K|\). For \(r\in[n]\), \(e_r\) denotes the \(r\)-th standard basis vector of \(\R^n\). The ordered diagonal variables \(Z_{(r)}\), the ordering permutation \(\pi\), and the ordered off-diagonal variables \(X_{ij}=Q_{\pi(i)\pi(j)}\) are as defined in Section 1.

\subsection{Almost-sure uniqueness}

The next lemma records a standard genericity property of continuous quadratic programs over a fixed polytope.

\begin{lemma}[Almost-sure uniqueness]
Under the GOE law, problem \eqref{eq:1.1} has a unique global minimizer almost surely.
\end{lemma}

\begin{proof}
For a nonempty support \(K\subseteq[n]\), let \(Q_K\) be the corresponding principal submatrix and let \(x_K\) denote the restriction of \(x\) to the coordinates in \(K\). If a minimizer has exact support \(K\), then the KKT conditions give
\begin{equation}
Q_Kx_K=\lambda \1_K,
\qquad \1_K^\top x_K=1,
\qquad x_K>0,
\label{eq:3.3}
\end{equation}
for some scalar \(\lambda\), together with the outside complementary inequalities
\begin{equation}
Q_{rK}x_K\ge \lambda,
\qquad r\notin K.
\label{eq:3.4}
\end{equation}
Outside an algebraic exceptional set on which \(Q_K\) is singular or \(\1_K^\top Q_K^{-1}\1_K=0\), the equalities in \eqref{eq:3.3} have at most one solution, namely
\begin{equation}
x_K=\frac{Q_K^{-1}\1_K}{\1_K^\top Q_K^{-1}\1_K},\qquad
\lambda_K=\frac1{\1_K^\top Q_K^{-1}\1_K}.
\end{equation}
Thus each fixed support produces at most one candidate optimizer, outside finitely many algebraic hypersurfaces.

If two distinct supports \(K\ne L\) both produced global minimizers, then their candidate values would satisfy \(\lambda_K=\lambda_L\), because each \(\lambda_K\) is the objective value of the corresponding candidate minimizer. Multiplying this equality by the nonzero denominators in the displayed formula above gives a polynomial equation in the entries of \(Q\). This polynomial is not identically zero. To see this, take a diagonal matrix \(D=\operatorname{diag}(d_1,\ldots,d_n)\) with \(d_i>0\) and with all reciprocal subset sums
\[
\sum_{i\in K}\frac1{d_i},\qquad \varnothing\ne K\subseteq[n],
\]
distinct; for instance choose \(1/d_i=2^i\). For this diagonal matrix,
\[
\lambda_K=\left(\sum_{i\in K}\frac1{d_i}\right)^{-1},
\]
so \(\lambda_K\ne\lambda_L\) whenever \(K\ne L\). Thus the equality polynomial for each fixed support pair is nonzero. Since the GOE law is absolutely continuous and there are only finitely many support pairs, the union of all such exceptional algebraic sets has probability zero. Hence the global minimizer is almost surely unique.
\end{proof}

From now on, write \(x^\star\) for the almost surely unique global optimizer and
\[
\kappa_n:=|\supp(x^\star)|.
\]

\subsection{Deterministic support conditions}

The following row-average necessary condition from Chen and Peng \cite{chen-peng} is the only deterministic fact needed to control solutions with support size at least three. We recall the short proof.

\begin{proposition}[Deterministic row-average necessary condition]\label{prop:det-support}
Suppose \(x^\star\) is the global optimizer of \eqref{eq:1.1}, and let \(K:=\supp(x^\star)\) with \(|K|=k>1\). Then there exists an active row \(i\in K\) such that
\begin{equation}
\frac1k\sum_{j\in K}Q_{ij}<\min_{1\le r\le n}Q_{rr}.
\label{eq:row-avg-cond}
\end{equation}
\end{proposition}

\begin{proof}
Let \(x_K\) be the restriction of \(x^\star\) to \(K\), and let \(\lambda=(x^\star)^\top Qx^\star\). Since \(|K|>1\), almost-sure uniqueness implies \(\lambda<\min_r Q_{rr}\); otherwise a vertex attaining the same value would also be optimal. The active KKT equations give \(Q_Kx_K=\lambda\1_K\). For \(i\in K\), put \(A_i=k^{-1}\sum_{j\in K}Q_{ij}\). By symmetry,
\[
\sum_{i\in K}x_i A_i
=\frac1k\sum_{j\in K}(Q_Kx_K)_j
=\lambda.
\]
Because \(x_i>0\) and \(\sum_{i\in K}x_i=1\), at least one active row satisfies \(A_i\le\lambda\). Combining this with \(\lambda<\min_r Q_{rr}\) gives \eqref{eq:row-avg-cond}.
\end{proof}

\section{Two-coordinate pair events}\label{sec:two-coordinate}

Recall that \(\pi\) orders the diagonal entries. Fix two diagonal order indices \(1\le i<j\le n\), and write
\[
a:=Z_{(i)},\qquad c:=Z_{(j)},\qquad b:=X_{ij}.
\]
For \(0\le t\le1\), the point \(t e_{\pi(i)}+(1-t)e_{\pi(j)}\) lies on the edge associated with \((i,j)\), and the objective restricted to this edge is
\begin{equation}
f_{ij}(t)=a t^2+2bt(1-t)+c(1-t)^2.
\label{eq:4.1}
\end{equation}

\begin{proposition}[Exact two-point threshold]
Let \(m\le a\le c\). Then
\begin{equation}
\min_{0\le t\le1}f_{ij}(t)<m
\quad\Longleftrightarrow\quad
b<m-\sqrt{(a-m)(c-m)}.
\label{eq:4.2}
\end{equation}
Consequently, taking \(m=Z_{(1)}\), there exists \(t\in[0,1]\) such that the point \(t e_{\pi(i)}+(1-t)e_{\pi(j)}\) has objective value below every vertex value if and only if
\begin{equation}
X_{ij}<\tau_{ij},
\label{eq:4.3}
\end{equation}
where
\begin{equation}
\tau_{ij}:=Z_{(1)}-\sqrt{(Z_{(i)}-Z_{(1)})(Z_{(j)}-Z_{(1)})}.
\label{eq:4.4}
\end{equation}
\end{proposition}

\begin{proof}
For \(t\in(0,1)\), set \(w=t/(1-t)\). Dividing \(f_{ij}(t)-m\) by \((1-t)^2\) gives the equivalent inequality
\[
(a-m)w^2+2(b-m)w+(c-m)<0,
\qquad w>0.
\]
The endpoints cannot give a strict value below \(m\), since \(f_{ij}(0)-m=c-m\ge0\) and \(f_{ij}(1)-m=a-m\ge0\). If the displayed expression is negative for some \(w>0\), then necessarily \(b<m\); otherwise all three coefficients are nonnegative.

First suppose \(a>m\). With \(b<m\), the quadratic has positive leading coefficient. Since its value at \(w=0\) is nonnegative, it is negative for some positive \(w\) if and only if its discriminant is positive:
\[
4(b-m)^2-4(a-m)(c-m)>0.
\]
Because \(b-m<0\), this is equivalent to
\[
b<m-\sqrt{(a-m)(c-m)}.
\]

It remains to handle the degenerate endpoint case \(a=m\). The inequality reduces to
\[
2(b-m)w+(c-m)<0,
\qquad w>0.
\]
This holds for some \(w>0\) if and only if \(b<m\). The right-hand side of \eqref{eq:4.2} also reduces to \(b<m\), because \(\sqrt{(a-m)(c-m)}=0\). This includes the subcase \(a=c=m\). Therefore \eqref{eq:4.2} holds in all cases.
\end{proof}

\begin{remark}[Comparison with the one-row relaxation]\label{rem:scale-vs-constant}
Since \(i<j\), we have \(Z_{(j)}\ge Z_{(i)}\), and hence
\[
\sqrt{(Z_{(i)}-Z_{(1)})(Z_{(j)}-Z_{(1)})}\ge Z_{(i)}-Z_{(1)}.
\]
Therefore
\[
\tau_{ij}
=Z_{(1)}-\sqrt{(Z_{(i)}-Z_{(1)})(Z_{(j)}-Z_{(1)})}
\le 2Z_{(1)}-Z_{(i)}.
\]
The one-row necessary condition used in Chen--Peng-type arguments is based on the linear threshold
\[
X_{ij}+Z_{(i)}\le2Z_{(1)},
\]
or equivalently \(X_{ij}\le 2Z_{(1)}-Z_{(i)}\).  This is a relaxation of the exact two-point threshold \(X_{ij}<\tau_{ij}\): it replaces the curved boundary determined by the product of two diagonal gaps by a linear boundary determined by one gap.  The two thresholds coincide only in the degenerate equal-gap limit \(Z_{(j)}=Z_{(i)}\).

For comparison, Appendix \ref{app:one-row} uses the one-row relaxation to develop probability asymptotics that recover the same order of magnitude as the non-vertex probability \(\Prob(\kappa_n>1)\). It also shows that the fixed-order part of the one-row first-moment bound has leading constant \(8\sqrt{2\pi}\), four times the exact leading constant \(2\sqrt{2\pi}\). The probability asymptotics established there may also be of independent interest.
\end{remark}

Define the two-coordinate event
\begin{equation}
F_{ij}:=\{X_{ij}<\tau_{ij}\},\qquad 1\le i<j\le n.
\label{eq:4.5}
\end{equation}
Thus \(F_{ij}\) is the event that the line segment associated with \((i,j)\) contains a point whose objective value is strictly smaller than the best vertex value, namely \(Z_{(1)}\). Its conditional probability is
\begin{equation}
p_{ij}:=\Prob(F_{ij}\mid Z_{(1)},\ldots,Z_{(n)})=\Phi(\sqrt2\tau_{ij}).
\label{eq:4.6}
\end{equation}
Conditioned on the full diagonal vector, equivalently on the order statistics together with the ordering permutation \(\pi\), the ordered off-diagonal variables \(X_{ij}=Q_{\pi(i)\pi(j)}\) are independent \(N(0,1/2)\). Since the off-diagonal entries are i.i.d. and independent of the diagonal entries, and since the thresholds \(\tau_{ij}\) depend only on the order statistics, the same product formula remains valid after conditioning only on \(Z_{(1)},\ldots,Z_{(n)}\).

\begin{proposition}[Pair-event product formula]
Conditionally on \(Z_{(1)},\ldots,Z_{(n)}\),
\begin{equation}
\Prob\left(\bigcup_{1\le i<j\le n}F_{ij}\,\middle|\,Z_{(1)},\ldots,Z_{(n)}\right)
=1-\prod_{1\le i<j\le n}(1-p_{ij}).
\label{eq:4.7}
\end{equation}
Consequently,
\begin{equation}
\Prob\left(\bigcup_{i<j}F_{ij}\right)
=\E\left[1-\prod_{i<j}(1-p_{ij})\right].
\label{eq:4.8}
\end{equation}
Moreover,
\begin{equation}
0\le
\Prob\left(\bigcup_{i<j}F_{ij}\right)-\Prob(\kappa_n=2)
\le \Prob(\kappa_n\ge3).
\label{eq:4.9}
\end{equation}
\end{proposition}

\begin{proof}
The product formula follows from the conditional-independence statement above, since each \(F_{ij}\) depends on only one ordered off-diagonal entry \(X_{ij}\). If \(\kappa_n=2\), then the global optimizer lies on some two-coordinate segment and has objective value strictly below every vertex, so the corresponding event \(F_{ij}\) occurs. Conversely, if some \(F_{ij}\) occurs, then the global optimum is strictly smaller than every vertex value, so \(\kappa_n\ne1\). Hence
\[
\{\kappa_n=2\}\subseteq\bigcup_{i<j}F_{ij}\subseteq\{\kappa_n>1\},
\]
which implies \eqref{eq:4.9}.
\end{proof}

Thus, to estimate \(\Prob(\kappa_n=2)\), it is enough to estimate the probability of a union of conditionally independent two-coordinate events, with an error bounded by \(\Prob(\kappa_n\ge3)\).

\section{Product decomposition}

For every \(j\ge2\), the threshold for the pair \((1,j)\) satisfies
\[
\tau_{1j}=Z_{(1)}.
\]
Thus the events \(F_{1j}\), \(j=2,\ldots,n\), all have the same conditional probability
\begin{equation}
q_n:=\Prob(F_{1j}\mid Z_{(1)},\ldots,Z_{(n)})=\Phi(\sqrt2 Z_{(1)}).
\label{eq:5.1}
\end{equation}
Factor the product in \eqref{eq:4.8} as
\begin{equation}
\prod_{1\le i<j\le n}(1-p_{ij})
=(1-q_n)^{n-1}\prod_{2\le i<j\le n}(1-p_{ij}).
\label{eq:5.2}
\end{equation}
The first factor, \((1-q_n)^{n-1}\), is the conditional probability that none of the pairs \((1,j)\), \(j\ge2\), produces an improvement over the best vertex. We therefore define
\begin{equation}
S_n:=\E\left[1-(1-q_n)^{n-1}\right].
\label{eq:5.3}
\end{equation}
Equivalently, after the diagonal order statistics have been fixed, this is the probability that at least one pair involving the smallest diagonal entry produces a point below the best vertex value.

A second factorization isolates pairs involving the second smallest diagonal entry. Define
\begin{equation}
A_n:=\E\left[(1-q_n)^{n-1}
\left(1-\prod_{j=3}^n(1-p_{2j})\right)\right].
\label{eq:5.4}
\end{equation}
This term accounts for pairs \((2,j)\), \(j\ge3\), on the event that no pair involving the smallest diagonal entry has already produced a point below the best vertex value. Finally define
\begin{equation}
B_n:=\E\left[(1-q_n)^{n-1}\prod_{j=3}^n(1-p_{2j})
\left(1-\prod_{3\le i<j\le n}(1-p_{ij})\right)\right].
\label{eq:5.5}
\end{equation}
This term accounts for the remaining pairs \((i,j)\) with \(3\le i<j\le n\), on the event that neither of the two preceding classes of pairs has occurred.

\begin{proposition}[Exact product decomposition]\label{prop:product-decomp}
For every \(n\ge3\),
\begin{equation}
\E\left[1-\prod_{1\le i<j\le n}(1-p_{ij})\right]
=S_n+A_n+B_n.
\label{eq:5.6}
\end{equation}
Consequently,
\begin{equation}
\Prob(\kappa_n=2)
=S_n+A_n+B_n+O(\Prob(\kappa_n\ge3)).
\label{eq:5.7}
\end{equation}
\end{proposition}

\begin{proof}
Factor
\[
\prod_{2\le i<j\le n}(1-p_{ij})
=
\left(\prod_{j=3}^n(1-p_{2j})\right)
\left(\prod_{3\le i<j\le n}(1-p_{ij})\right).
\]
Then apply \(1-uv=(1-u)+u(1-v)\), multiply by \((1-q_n)^{n-1}\), and take expectations. Equation \eqref{eq:5.7} follows from \eqref{eq:4.9}.
\end{proof}

Equation \eqref{eq:5.7} estimates the support-two probability by three explicit two-coordinate contributions, with the only remaining error controlled by \(\Prob(\kappa_n\ge3)\). After Section \ref{sec:support-three} proves that this error is negligible on the \(B_n\) scale, the same expansion also gives the probability \(\Prob(\kappa_n>1)\) that the optimizer is not a vertex. Here \(S_n\) is the contribution from edges incident to the best diagonal vertex, \(A_n\) is the contribution from edges incident to the second-best diagonal vertex after the first class of edges has failed, and \(B_n\) is the remaining contribution involving neither of the two lowest diagonal entries.

\subsection{A product-linearization estimate}

The same elementary product estimate is used in three places: for \(S_n\), where all factors are equal to \(q_n\); for \(A_n\), where the factors are \(p_{2j}\); and for \(B_n\), where the factors are \(p_{ij}\) with \(3\le i<j\le n\). We state it once to avoid repeating the same union-product analysis.

\begin{lemma}[Product linearization]\label{lem:product-linearization}
Let \(0\le a_1,\ldots,a_m\le1\), set \(S:=\sum_{r=1}^m a_r\), and
\[
T:=1-
\prod_{r=1}^m(1-a_r).
\]
Then
\begin{equation}\label{eq:product-linearization}
S-\frac12S^2\le T\le S.
\end{equation}
In particular,
\begin{equation}\label{eq:product-linearization-error}
0\le S-T\le \frac12S^2.
\end{equation}
\end{lemma}

\begin{proof}
The identity
\[
T=\sum_{r=1}^m a_r\prod_{s<r}(1-a_s)
\]
gives \(T\le S\). It also gives
\[
S-T=\sum_{r=1}^m a_r\left(1-\prod_{s<r}(1-a_s)\right)
\le \sum_{r=1}^m a_r\sum_{s<r}a_s
\le \frac12S^2,
\]
where the middle inequality applies the bound \(T\le S\) to the partial product. This proves both displays.
\end{proof}

\section{The minimum-diagonal contribution}\label{sec:S}

Let
\[
U_{(1)}:=\Phi(Z_{(1)}).
\]
Then \(U_{(1)}\) has the \(\operatorname{Beta}(1,n)\) distribution, and
\begin{equation}
q_n=\Psi(U_{(1)}).
\label{eq:6.1}
\end{equation}

Here \(\Psi\) is the GOE off-diagonal tail-transfer function from Corollary \ref{cor:goe-tail-transfer}. We shall use the asymptotic formula \eqref{eq:goe-tail-transfer} and the uniform upper bound \eqref{eq:goe-tail-transfer-upper} repeatedly in this section.

\begin{lemma}[Moments of the minimum-diagonal pair probability]\label{lem:qn-moments}
As \(n\to\infty\),
\begin{equation}
\E q_n=2\sqrt{2\pi}\frac{\sqrt{\log n}}{n^2}(1+o(1)),
\label{eq:6.4}
\end{equation}
and
\begin{equation}
\E q_n^2\le C\frac{\log n}{n^4}
\label{eq:6.5}
\end{equation}
for all sufficiently large \(n\).
\end{lemma}

\begin{proof}
Since \(U_{(1)}\) has density \(n(1-u)^{n-1}\),
\[
\E q_n=n\int_0^1(1-u)^{n-1}\Psi(u)\dd u.
\]
The part over \([1/2,1]\) is exponentially small. On \((0,1/2]\), Corollary \ref{cor:goe-tail-transfer} and Lemma \ref{lem:beta-log} with \(r=2\), \(\gamma=1/2\), give
\[
\E q_n
=\sqrt{2\pi}\,\Gamma(3)n^{-2}(\log n)^{1/2}(1+o(1))
=2\sqrt{2\pi}\frac{\sqrt{\log n}}{n^2}(1+o(1)).
\]
Similarly, \eqref{eq:goe-tail-transfer-upper} gives
\[
\E q_n^2
\le Cn\int_0^{1/2}(1-u)^{n-1}u^4\log(1/u)\dd u+O(e^{-cn}).
\]
Applying Lemma \ref{lem:beta-log} with \(r=4\) and \(\gamma=1\) yields \(\E q_n^2\le C(\log n)n^{-4}\).
\end{proof}

\begin{proposition}[Minimum-diagonal contribution]\label{prop:S-asympt}
Define
\begin{equation}
S_n:=\E\left[1-(1-q_n)^{n-1}\right].
\label{eq:6.6}
\end{equation}
Then
\begin{equation}
S_n=2\sqrt{2\pi}\frac{\sqrt{\log n}}n(1+o(1)).
\label{eq:6.7}
\end{equation}
\end{proposition}

\begin{proof}
Apply Lemma \ref{lem:product-linearization} to the \(n-1\) identical factors \(a_r=q\). Then
\[
0\le (n-1)q-\{1-(1-q)^{n-1}\}\le \frac12(n-1)^2q^2.
\]
Taking \(q=q_n\), then taking expectations and applying Lemma \ref{lem:qn-moments} gives
\[
S_n=(n-1)\E q_n+O(n^2\E q_n^2)
=2\sqrt{2\pi}\frac{\sqrt{\log n}}n(1+o(1)),
\]
because \(n^2\E q_n^2=O((\log n)/n^2)=o(\sqrt{\log n}/n)\).
\end{proof}

\section{The term \texorpdfstring{\(A_n\)}{An}}\label{sec:A}

This section proves
\begin{equation}
A_n\sim 3\sqrt{\frac\pi2}\frac1{n\sqrt{\log n}},
\label{eq:7.1}
\end{equation}
where \(A_n\) is smaller than the contribution \(S_n\) from the smallest diagonal entry by a factor \(1/\log n\) and represents the rank-two correction term: it counts edges from the second-smallest diagonal vertex to the remaining vertices after the minimum-diagonal star has been removed. To evaluate \(A_n\), and later \(B_n\), we must integrate over the event that the next diagonal order statistics are close enough to the absolute minimum to affect the two-coordinate threshold. The Gaussian tail concentrates the main contribution in a boundary layer of relative width \(O(1/\log n)\). We therefore introduce the scaling \(v=x(1+\beta/s^2)/n\), with \(s=s_n(x)\) and \(\beta=O(1)\).

\subsection{Averaged edge-probability function}

Let
\begin{equation}
S_{2,n}:=\sum_{j=3}^n p_{2j},
\label{eq:7.2}
\end{equation}
with \(p_{ij}\) as defined in \eqref{eq:4.6}.

\begin{lemma}[Reduction of \(A_n\) to its linear intensity]\label{lem:A-linearization}
As \(n\to\infty\),
\begin{equation}
A_n=\E S_{2,n}+O\left(\frac{\log n}{n^2}\right).
\label{eq:7.3}
\end{equation}
\end{lemma}

\begin{proof}
Let
\[
T_{2,n}:=1-\prod_{j=3}^n(1-p_{2j}).
\]
By Lemma \ref{lem:product-linearization},
\[
0\le S_{2,n}-T_{2,n}\le \frac12S_{2,n}^2.
\]
Since \(\tau_{2j}\le Z_{(1)}\), we have \(p_{2j}\le q_n\), and hence
\(S_{2,n}\le(n-2)q_n\). Therefore,
\[
\E|S_{2,n}-T_{2,n}|
\le \frac12(n-2)^2\E q_n^2
=O\left(\frac{\log n}{n^2}\right).
\]
This error is \(o(n^{-1}(\log n)^{-1/2})\), the scale on which \(A_n\) is evaluated in Theorem \ref{thm:I-A}; hence the product and its linearized sum have the same first-order expectation. Moreover,
\[
0\le1-(1-q_n)^{n-1}\le(n-1)q_n,
\]
and \(T_{2,n}\le S_{2,n}\le(n-2)q_n\). Hence
\[
\E\left|\bigl((1-q_n)^{n-1}-1\bigr)T_{2,n}\right|
\le (n-1)(n-2)\E q_n^2
=O\left(\frac{\log n}{n^2}\right).
\]
Finally,
\[
A_n-\E S_{2,n}
=\E\{(1-q_n)^{n-1}T_{2,n}-S_{2,n}\}.
\]
The absolute value of the right-hand side is bounded by the sum of the two error terms just estimated, which proves \eqref{eq:7.3}.
\end{proof}

Let \(U_{(r)}:=\Phi(Z_{(r)})\). For \(0<u<v<w<1\), define
\begin{equation}
p(u,v,w):=\Phi\left(\sqrt2\left[\Phi^{-1}(u)-
\sqrt{(\Phi^{-1}(v)-\Phi^{-1}(u))(\Phi^{-1}(w)-\Phi^{-1}(u))}\right]\right).
\label{eq:7.5}
\end{equation}
Define the averaged edge-probability function
\begin{equation}
G(u,v):=\frac1{1-v}\int_v^1 p(u,v,w)\dd w,
\qquad 0<u<v<1.
\label{eq:7.6}
\end{equation}
When boundary-layer estimates below are stated at \(\beta=0\), so that \(v=u\), the values of \(G\), \(H\), and \(p(u,v,v)\) are understood by continuous extension from \(v>u\).

Set
\begin{equation}
I_n:=(n-2)\E G(U_{(1)},U_{(2)}).
\label{eq:7.9}
\end{equation}

\begin{proposition}[Averaged-function representation]\label{prop:G-representation}
We have
\begin{equation}
\E S_{2,n}=I_n=(n-2)\E G(U_{(1)},U_{(2)}).
\label{eq:7.7}
\end{equation}
Equivalently,
\begin{equation}
\E S_{2,n}
=n(n-1)(n-2)\int_0^1\int_0^v(1-v)^{n-2}G(u,v)\dd u\dd v.
\label{eq:7.8}
\end{equation}
\end{proposition}

\begin{proof}
Conditioned on \((U_{(1)},U_{(2)})=(u,v)\), the unordered remaining uniforms are independent and uniformly distributed on \((v,1)\). The summands \(p_{2j}\), \(j=3,\ldots,n\), are therefore exchangeable under this conditional law. Consider any index \(j\ge3\). The corresponding uniform variable takes a value \(w\in(v,1)\). The conditional probability that the pair \((2,j)\) satisfies the threshold is \(p(u,v,w)\). Averaging this probability over \(w\) gives
\[
\frac1{1-v}\int_v^1p(u,v,w)\dd w=G(u,v).
\]
There are \(n-2\) such indices, so linearity of conditional expectation gives
\[
\E(S_{2,n}\mid U_{(1)}=u,U_{(2)}=v)=(n-2)G(u,v).
\]
Taking expectation gives \eqref{eq:7.7}. The density of \((U_{(1)},U_{(2)})\) is \(n(n-1)(1-v)^{n-2}\1_{0<u<v<1}\), which gives \eqref{eq:7.8}.
\end{proof}

We now evaluate \(I_n\).

\subsection{Boundary-layer asymptotics for the averaged edge probability}

Recall the quantile and logarithmic scales \(s_n(x)\), \(L_n(x)\), and \(\ell_n(x)\) from Lemma \ref{lem:quantile-scale}. The leading contribution comes from the range where the second order statistic is close to the minimum on the scale \(v/u-1=O(1/\log n)\). We write this range as
\begin{equation}
 u=\frac xn,
 \qquad
 v=\frac{x(1+\beta/s^2)}n,
 \qquad s=s_n(x),
\label{eq:7.11}
\end{equation}
where we use \(\beta\) to rescale the boundary layer.

For later use we record the common integral representation. Given \(0<u<v<1\), write
\[
s=-\Phi^{-1}(u),\qquad t=-\Phi^{-1}(v),\qquad h=s-t.
\]
Changing variables \(w=\Phi(z)\) in \eqref{eq:7.6} gives
\begin{equation}
G(u,v)=\frac1{1-v}\int_{-t}^\infty
\Phi\left(-\sqrt2\left[s+\sqrt{h(z+s)}\right]\right)\phi(z)\dd z.
\label{eq:7.13}
\end{equation}
We split this integral at \(-s/2\):
\[
G(u,v)=G_-(u,v)+G_+(u,v),
\]
where \(G_-\) and \(G_+\) denote the contributions from
\[
D_-:=[-t,\infty)\cap(-\infty,-s/2),
\qquad
D_+:=[-t,\infty)\cap[-s/2,\infty),
\]
respectively. The set \(D_-\) is interpreted as empty when \(-t\ge -s/2\). Both the boundary limit and the global bounds below are consequences of this representation, Mills' inequality, and Lemma \ref{lem:quantile-gap}. The following lemma records the estimates for the two parts of the split.

\begin{lemma}[Basic split bounds for \(G\)]\label{lem:G-split-basic}
In the notation of \eqref{eq:7.13}, with \(0<u<v<1\) and \(G=G_-+G_+\) split at \(-s/2\), the following bounds hold for all sufficiently large \(s=-\Phi^{-1}(u)\):
\begin{align}
G_+(u,v)
&\le C\frac{e^{-s^2}}s\exp\{-\sqrt2\,s\sqrt{sh}\},
\label{eq:G-basic-plus}\\
G_-(u,v)
&\le C\frac{e^{-s^2}}{s^2}\exp\{-s^2/8-2sh\}.
\label{eq:G-basic-minus}
\end{align}
\end{lemma}

\begin{proof}
For \(G_+\), Mills' inequality and \(z+s\ge s/2\) give the bound
\begin{equation}
\Phi\left(-\sqrt2\left[s+\sqrt{h(z+s)}\right]\right)
\le \frac{C}{s}\exp\{-s^2-2s\sqrt{h(z+s)}\}
\le \frac{C}{s}e^{-s^2}\exp\{-\sqrt2\,s\sqrt{sh}\}.
\label{eq:common-mills-G}
\end{equation}
After integrating against \(\phi(z)\dd z\), the factor \((1-v)^{-1}\) is harmless because
\[
(1-v)^{-1}\int_{-t}^{\infty}\phi(z)\dd z=1.
\]
This proves \eqref{eq:G-basic-plus}.

For \(G_-\), the interval is empty unless \(t>s/2\). In the nonempty case,
\(1-v=\Phi(t)\ge\Phi(s/2)\ge1/2\) for all sufficiently large \(s\). Also, for \(z\in[-t,-s/2]\),
\(z+s\ge h\). Applying Mills' inequality on this interval yields
\[
\Phi\left(-\sqrt2\left[s+\sqrt{h(z+s)}\right]\right)
\le \frac{C}{s}\exp\{-s^2-2sh\}.
\]
Together with the Gaussian tail estimate
\[
\int_{-\infty}^{-s/2}\phi(z)\dd z=\Phi(-s/2)\le Cs^{-1}e^{-s^2/8},
\]
this gives
\[
G_-(u,v)\le C\frac{e^{-s^2}}{s^2}e^{-s^2/8-2sh},
\]
which is \eqref{eq:G-basic-minus}.
\end{proof}

\begin{lemma}[Boundary-layer limit]\label{lem:G-boundary-layer}
Fix \(0<\delta<M<\infty\) and \(B<\infty\). Uniformly for \((x,\beta)\in[\delta,M]\times[0,B]\), with \(s=s_n(x)\),
\begin{equation}
\frac{n^2}{s}G\left(\frac xn,\frac{x(1+\beta/s^2)}n\right)
\longrightarrow
\sqrt\pi\,x^2 e^{-2\sqrt\beta}.
\label{eq:7.12}
\end{equation}
\end{lemma}

\begin{proof}
In \eqref{eq:7.13}, take \(u=x/n\), \(v=x(1+\beta/s^2)/n\), and \(s=s_n(x)\). Lemma \ref{lem:quantile-gap}\textup{(i)} gives
\begin{equation}
 s^3h\to\beta
\label{eq:7.14}
\end{equation}
uniformly on compact \((x,\beta)\)-sets. For fixed \(z\), \eqref{eq:7.14} implies
\[
h(z+s)=\frac{\beta+o(1)}{s^2},
\]
and hence
\begin{equation}
\left[s+\sqrt{h(z+s)}\right]^2=s^2+2\sqrt\beta+o(1).
\label{eq:7.15}
\end{equation}
Mills' ratio applied to the Gaussian distribution function in \eqref{eq:7.13} gives
\[
\Phi\left(-\sqrt2\left[s+\sqrt{h(z+s)}\right]\right)
=\frac{1+o(1)}{2\sqrt\pi s}
\exp\left(-\left[s+\sqrt{h(z+s)}\right]^2\right).
\]
Combining this with \eqref{eq:sn-exp-id} gives pointwise convergence of the integrand, after multiplication by \(n^2/s\), to
\[
\sqrt\pi\,x^2e^{-2\sqrt\beta}\phi(z).
\]

It remains to pass the limit through the \(z\)-integral. Extend the integrand in \eqref{eq:7.13} by zero below the moving lower limit \(-t\). On \(z\ge -s/2\), \eqref{eq:common-mills-G}, the inequality \(s+\sqrt{h(z+s)}\ge s\), and Lemma \ref{lem:quantile-scale} give the single bound
\[
\frac{n^2}{s}
\Phi\left(-\sqrt2\left[s+\sqrt{h(z+s)}\right]\right)\phi(z)
\le C\phi(z),
\]
uniformly on the fixed compact set, where \(x\) is bounded away from both \(0\) and \(\infty\). For \(G_-\), the basic split estimate \eqref{eq:G-basic-minus} gives
\[
\frac{n^2}{s}G_-(u,v)
\le C\frac{n^2e^{-s^2}}{s^3}e^{-s^2/8},
\]
which tends to zero uniformly on compact \((x,\beta)\)-sets, again by Lemma \ref{lem:quantile-scale}. The dominated convergence theorem, for instance in Billingsley \cite{billingsley}, proves \eqref{eq:7.12}.
\end{proof}

\begin{lemma}[Global bounds for \(G\)]\label{lem:G-global-bounds}
Fix \(\eta\in(0,1]\). There exist constants \(C<\infty\) and \(c>0\) such that, for all sufficiently large \(n\), the following bounds hold.
\begin{enumerate}[label=(\roman*)]
\item If \(0<x\le\sqrt n\), \(s=s_n(x)\), and \(0\le\beta\le\eta s^2\), then
\begin{equation}
G\left(\frac xn,\frac{x(1+\beta/s^2)}n\right)
\le C\frac{x^2s}{n^2}e^{-c\sqrt\beta}.
\label{eq:7.16}
\end{equation}
\item If \(0<x\le\sqrt n\), \((1+\eta)x\le y<n\), \(s=s_n(x)\), and \(a=\log(y/x)\), then the far-ratio split \(G=G_++G_-\) satisfies
\begin{align}
G_+\left(\frac xn,\frac yn\right)
&\le C\frac{x^2s}{n^2}
   \exp\{-c_+s\sqrt a\},
\label{eq:7.17-plus}\\
G_-\left(\frac xn,\frac yn\right)
&\le C\frac{x^2s}{n^2}e^{-s^2/8}\left(\frac xy\right)^{6/5}.
\label{eq:7.17-minus}
\end{align}
Here \(c_+\) may be chosen as any fixed number in \( (1/\sqrt2,\sqrt{6/5}) \). Consequently, after possibly decreasing the constant, the full \(G\) also satisfies
\begin{equation}
G\left(\frac xn,\frac yn\right)
\le C\frac{x^2s}{n^2}
\exp\{-c s\sqrt{\log(y/x)}\},
\qquad c>1/\sqrt2.
\label{eq:7.17}
\end{equation}
\end{enumerate}
\end{lemma}

\begin{proof}
We use the basic split bounds in Lemma \ref{lem:G-split-basic} together with the quantile-gap estimates.

For part \textup{(i)}, set \(v=x(1+\beta/s^2)/n\). Lemma \ref{lem:quantile-gap}\textup{(ii)} gives \(h\ge c_1\beta/s^3\). The \(G_+\) estimate \eqref{eq:G-basic-plus} therefore gives
\[
G_+\left(\frac xn,v\right)
\le C\frac{e^{-s^2}}s e^{-c_2\sqrt\beta}.
\]
For \(G_-\), \eqref{eq:G-basic-minus} gives
\[
G_-\left(\frac xn,v\right)
\le C\frac{e^{-s^2}}{s^2}e^{-s^2/8}.
\]
Since \(0<x\le\sqrt n\) implies \(s\to\infty\) uniformly, and since \(0\le\beta\le\eta s^2\), the factor \(e^{-s^2/8}\) is bounded by \(C_\eta e^{-c_3\sqrt\beta}\) after decreasing \(c_3\). Applying \eqref{eq:sn-exp-upper} to both estimates and adding them proves \eqref{eq:7.16}.

For part \textup{(ii)}, put \(a=\log(y/x)\). Lemma \ref{lem:quantile-gap}\textup{(iii)} gives \(sh\ge c_0a\) with \(c_0=3/5\). The \(G_+\) bound \eqref{eq:G-basic-plus} yields
\[
G_+\left(\frac xn,\frac yn\right)
\le C\frac{e^{-s^2}}s\exp\{-\sqrt{2c_0}\,s\sqrt a\}.
\]
Because \(\sqrt{2c_0}=\sqrt{6/5}>1/\sqrt2\), \eqref{eq:sn-exp-upper} proves \eqref{eq:7.17-plus} for any fixed \(c_+\in(1/\sqrt2,\sqrt{6/5})\).

For \(G_-\), \eqref{eq:G-basic-minus}, \eqref{eq:sn-exp-upper}, and \(sh\ge(3/5)a\) give
\[
G_-\left(\frac xn,\frac yn\right)
\le C\frac{x^2s}{n^2}e^{-s^2/8}e^{-(6/5)a}
= C\frac{x^2s}{n^2}e^{-s^2/8}\left(\frac xy\right)^{6/5},
\]
which is \eqref{eq:7.17-minus}.

It remains only to convert the displayed \(G_-\) bound into the same strong exponential form as the \(G_+\) bound. The elementary arithmetic-geometric mean inequality
\[
\frac{s^2}{8}+\frac65a
\ge \sqrt{\frac35}\,s\sqrt a,
\qquad a>0,
\]
shows that \eqref{eq:7.17-minus} also implies a bound of the form \(C x^2s n^{-2}\exp\{-c s\sqrt a\}\) with some \(c>1/\sqrt2\). Together with the already proved \(G_+\) estimate \eqref{eq:7.17-plus}, this proves the full bound \eqref{eq:7.17}.
\end{proof}

\begin{lemma}[Far-ratio exponential integral]\label{lem:far-exp-integral}
Fix \(\eta\in(0,1]\) and \(c_+>1/\sqrt2\). There exist constants \(C,c,\rho>0\) such that, uniformly for \(0<x\le\sqrt n\), with \(s=s_n(x)\),
\begin{equation}
\int_{(1+\eta)x}^{n}
\exp\{-c_+s\sqrt{\log(y/x)}\}\,\dd y
\le Cx\ell_n(x)\left(e^{-cs}+\left(\frac{x}{n}\right)^\rho\right).
\label{eq:far-exp-integral}
\end{equation}
\end{lemma}

\begin{proof}
Put \(L=L_n(x)=\log(n/x)\) and \(w=\sqrt{\log(y/x)}\). Then \(0<x\le\sqrt n\) implies \(L\ge(1/2)\log n\), and
\[
y=xe^{w^2},\qquad \dd y=2xwe^{w^2}\,\dd w.
\]
The lower and upper limits are
\[
w_0=\sqrt{\log(1+\eta)},\qquad w_1=\sqrt L.
\]
Therefore the integral on the left-hand side of \eqref{eq:far-exp-integral} equals
\[
2x\int_{w_0}^{w_1}w\exp\{w^2-c_+sw\}\,\dd w.
\]
The exponent \(w^2-c_+sw\) is convex, so its maximum on \([w_0,w_1]\) occurs at an endpoint. At the lower endpoint,
\[
w_0^2-c_+sw_0\le -c_1s
\]
for some \(c_1>0\). At the upper endpoint, \eqref{eq:sn-log-refined} gives \(s\ge(1-o(1))\sqrt{2L}\) uniformly for \(0<x\le\sqrt n\). Since \(c_+\sqrt2>1\), there is \(\rho>0\) such that, for all large \(n\),
\[
w_1^2-c_+sw_1=L-c_+s\sqrt L\le -\rho L.
\]
Since \(w\le w_1=\sqrt L\) and the interval length is at most \(\sqrt L\), the factor \(w\) in the integrand contributes at most \(CL\) after taking the endpoint bound. Thus
\[
2x\int_{w_0}^{w_1}w\exp\{w^2-c_+sw\}\,\dd w
\le CxL\{e^{-c_1s}+e^{-\rho L}\}.
\]
Because \(e^{-\rho L}=(x/n)^\rho\) and \(L\le \ell_n(x)\), this proves \eqref{eq:far-exp-integral}.
\end{proof}

The next lemma integrates the far-ratio estimates for \(G_+\) and \(G_-\) in the second order-statistic variable. Later proofs use only this integrated form.

\begin{lemma}[Integrated far-ratio bound for \(G\)]\label{lem:G-far-integrated}
Fix \(\eta\in(0,1]\). There exist constants \(C,c,\rho>0\) such that, for all sufficiently large \(n\), if \(0<x\le\sqrt n\) and \(s=s_n(x)\), then
\begin{equation}\label{eq:G-far-integrated}
\int_{(1+\eta)x}^n G\left(\frac xn,\frac yn\right)\dd y
\le C\frac{x^3s}{n^2}
\left[\ell_n(x)\left(e^{-cs}+\left(\frac{x}{n}\right)^\rho\right)+e^{-s^2/8}\right].
\end{equation}
\end{lemma}

\begin{proof}
Split the integral into its \(G_+\) and \(G_-\) parts. For \(G_+\), Lemma \ref{lem:G-global-bounds}\textup{(ii)} gives, for some \(c_+>1/\sqrt2\),
\[
G_+\left(\frac xn,\frac yn\right)
\le C\frac{x^2s}{n^2}\exp\{-c_+s\sqrt{\log(y/x)}\}.
\]
Lemma \ref{lem:far-exp-integral} therefore bounds the integrated \(G_+\) contribution by
\[
C\frac{x^3s\ell_n(x)}{n^2}
\left(e^{-cs}+\left(\frac{x}{n}\right)^\rho\right).
\]
For \(G_-\), Lemma \ref{lem:G-global-bounds}\textup{(ii)} gives
\[
G_-\left(\frac xn,\frac yn\right)
\le C\frac{x^2s}{n^2}e^{-s^2/8}\left(\frac{x}{y}\right)^{6/5}.
\]
Since
\[
\int_{(1+\eta)x}^{\infty}\left(\frac{x}{y}\right)^{6/5}\dd y
= x\int_{1+\eta}^{\infty}r^{-6/5}\dd r\le Cx,
\]
the integrated \(G_-\) contribution is at most \(C x^3s n^{-2}e^{-s^2/8}\). Adding the two bounds proves \eqref{eq:G-far-integrated}.
\end{proof}

\begin{theorem}[Asymptotics of the averaged function for \(A_n\)]\label{thm:I-A}
As \(n\to\infty\),
\begin{equation}
I_n=(n-2)\E G(U_{(1)},U_{(2)})
=\left(3\sqrt{\frac\pi2}+o(1)\right)\frac1{n\sqrt{\log n}}.
\label{eq:7.18}
\end{equation}
\end{theorem}

\begin{proof}
From \eqref{eq:7.8}, after the change of variables \(u=x/n\), \(v=y/n\),
\begin{equation}
I_n=\frac{(n-1)(n-2)}n
\int_0^n\int_x^n\left(1-\frac yn\right)^{n-2}G\left(\frac xn,\frac yn\right)\dd y\dd x.
\label{eq:7.19}
\end{equation}
The region \(x>\sqrt n\) is negligible because \(G\le1\) and the factor \((1-y/n)^{n-2}\) is exponentially small there. On \(0<x\le\sqrt n\), split into a near region \(x\le y\le(1+\eta)x\) and a far-ratio region \(y\ge(1+\eta)x\). Let \(I_{n,\mathrm{far}}\) denote the corresponding part of the integral in \eqref{eq:7.19} over \(0<x\le\sqrt n\) and \((1+\eta)x\le y<n\). Using \((1-y/n)^{n-2}\le Ce^{-c_1y}\le Ce^{-c_1x}\) and Lemma \ref{lem:G-far-integrated},
\begin{align*}
I_{n,\mathrm{far}}
&\le Cn\int_0^{\sqrt n}e^{-c_1x}
\left[\int_{(1+\eta)x}^nG\left(\frac xn,\frac yn\right)\dd y\right]\dd x\\
&\le \frac{C}{n}\int_0^{\sqrt n}x^3s_n(x)e^{-c_1x}
\left[\ell_n(x)\left(e^{-cs_n(x)}+\left(\frac{x}{n}\right)^\rho\right)+e^{-s_n(x)^2/8}\right]\dd x.
\end{align*}
By Lemma \ref{lem:outer-integrals}, the three bracketed terms are respectively smaller than any inverse power of \(\log n\), \(O(n^{-\rho})\), and \(O(n^{-1/4})\), up to harmless powers of \(\log n\). Hence
\[
I_{n,\mathrm{far}}=o\left(\frac1{n\sqrt{\log n}}\right),
\]
so the far-ratio region is negligible for \(I_n\).

In the near region set \(s=s_n(x)\) and
\[
y=x\left(1+\frac\beta{s^2}\right),
\qquad \dd y=\frac{x}{s^2}\,\dd\beta.
\]
For \((x,\beta)\) restricted to a compact rectangle \([\delta,M]\times[0,B]\), Lemma \ref{lem:G-boundary-layer} gives
\[
\frac{n^2}{s}G\left(\frac xn,\frac{x(1+\beta/s^2)}n\right)
\to \sqrt\pi x^2e^{-2\sqrt\beta},
\]
while
\[
\left(1-\frac{x(1+\beta/s^2)}n\right)^{n-2}\to e^{-x},
\qquad
\frac{\sqrt{\log n}}{s}\to\frac1{\sqrt2}.
\]
Including the Jacobian \(\dd y=x s^{-2}\dd\beta\) and the prefactor in \eqref{eq:7.19}, the rescaled integrand, after multiplication by \(n\sqrt{\log n}\), converges on the compact rectangle to
\[
\sqrt{\frac\pi2}\,x^3e^{-x}e^{-2\sqrt\beta}.
\]
This proves the limit over each compact rectangle \(x\in[\delta,M]\), \(\beta\in[0,B]\).  We now pass from compact rectangles to the full integration region.  In the near region \(x\le y\le(1+\eta)x\), the change of variables above gives
\[
0\le \beta\le \eta s^2,
\]
so the upper \(\beta\)-limit grows like \(\log n\).  Lemma \ref{lem:G-global-bounds}\textup{(i)}, the bound \((1-y/n)^{n-2}\le e^{-c_1x}\) for \(0<x\le\sqrt n\), and the Jacobian \(\dd y=x s^{-2}\,\dd\beta\) show that the rescaled integrand in \eqref{eq:7.19}, after multiplication by \(n\sqrt{\log n}\), is bounded by
\[
C x^3 e^{-c_1x}e^{-c_2\sqrt\beta}\dd x\dd\beta.
\]
Here \(c_1,c_2>0\), and the factor \(\sqrt{\log n}/s\) is bounded by Lemma \ref{lem:quantile-scale}. The displayed function is integrable on \((0,\infty)^2\), uniformly in \(n\). This gives the standard truncate-then-dominated-convergence argument: for any \(\varepsilon>0\), choose \(\delta,M,B\) so that the integral of the displayed bound over \(\{x<\delta\}\cup\{x>M\}\cup\{\beta>B\}\) is below \(\varepsilon\), and then apply the dominated convergence theorem on \([\delta,M]\times[0,B]\). In particular, the growing near-region interval \(0\le\beta\le\eta s^2\) causes no difficulty because the tail \(\beta>B\) is controlled by \(e^{-c_2\sqrt\beta}\). The complementary far-ratio region was shown negligible above. Therefore
\[
\lim_{n\to\infty}n\sqrt{\log n}\,I_n
=\sqrt{\frac\pi2}
\int_0^\infty x^3e^{-x}\dd x
\int_0^\infty e^{-2\sqrt\beta}\dd\beta.
\]
Finally,
\[
\int_0^\infty x^3e^{-x}\dd x=6,
\qquad
\int_0^\infty e^{-2\sqrt\beta}\dd\beta=\frac12.
\]
Thus the limit equals \(3\sqrt{\pi/2}\).
\end{proof}

Combining Lemma \ref{lem:A-linearization}, Proposition \ref{prop:G-representation}, and Theorem \ref{thm:I-A} yields the desired asymptotic formula for \(A_n\).

\begin{corollary}[Asymptotics of \(A_n\)]\label{cor:A-asympt}
\begin{equation}
A_n=\left(3\sqrt{\frac\pi2}+o(1)\right)\frac1{n\sqrt{\log n}}.
\label{eq:7.20}
\end{equation}
\end{corollary}

\begin{remark}[No finite block carries \(A_n\)]
For every fixed integer \(M\ge1\),
\[
\sum_{j=3}^{M+2}\E p_{2j}
\le M\E q_n
=O\left(\frac{\sqrt{\log n}}{n^2}\right)
=o\left(\frac1{n\sqrt{\log n}}\right).
\]
Thus \(A_n\) is not produced by finitely many remaining diagonal order indices. Its mass is spread over the boundary layer captured by the averaged edge-probability function \(G\).
\end{remark}

\section{The term \texorpdfstring{\(B_n\)}{Bn}}\label{sec:B}

This section proves
\begin{equation}
B_n\sim 9\sqrt{\frac\pi2}\frac1{n(\log n)^{3/2}}.
\label{eq:8.1}
\end{equation}
Set
\[
S_{3,n}:=\sum_{3\le i<j\le n}p_{ij}.
\]
We first evaluate the linear intensity \(\E S_{3,n}\). We then use Lemma \ref{lem:product-linearization} to show that the higher-order terms in the product defining \(B_n\) are negligible on the scale of \(B_n\).

\begin{lemma}[Monotonicity of the pair probability]
For fixed \(u\in(0,1)\), the map \((v,w)\mapsto p(u,v,w)\) is nonincreasing in each coordinate on \(u<v<w<1\). In particular, if \(u<v\le x\le y<1\), then
\[
p(u,x,y)\le p(u,v,y)\le p(u,v,v)\le \Psi(u).
\]
\end{lemma}

\begin{proof}
The argument of \(\Phi\) in \eqref{eq:7.5} decreases when either \(v\) or \(w\) increases, because \(\Phi^{-1}\) is increasing and the square-root term increases. The final inequality follows by taking the limiting value \(v=w=u\).
\end{proof}

\subsection{Linear intensity of the remaining pairs}

Define the two-sample averaged edge-probability function
\begin{equation}
H(u,v):=\frac{2}{(1-v)^2}\int_v^1\int_w^1 p(u,w,z)\dd z\dd w,
\qquad 0<u<v<1.
\label{eq:8.2}
\end{equation}
At boundary-layer points with \(v=u\), this notation denotes the continuous
extension from \(v>u\).

Conditioned on \((U_{(1)},U_{(2)})=(u,v)\), the remaining \(n-2\) uniforms are i.i.d. on \((v,1)\), and hence
\begin{equation}
\E(S_{3,n}\mid U_{(1)}=u,U_{(2)}=v)
=\binom{n-2}{2}H(u,v).
\label{eq:8.3}
\end{equation}
Moreover,
\begin{equation}
H(u,v)=\frac{2}{(1-v)^2}\int_v^1(1-w)G(u,w)\dd w.
\label{eq:8.4}
\end{equation}

\begin{lemma}[Bounds for the two-sample average \(H\)]\label{lem:H-bound}
Fix \(\eta\in(0,1]\). There exist constants \(C,c,\rho>0\) such that the following bounds hold for all sufficiently large \(n\). Let \(s=s_n(x)\).
\begin{enumerate}[label=(\roman*)]
\item If \(0<x\le\sqrt n\), \(0\le\beta\le\eta s^2\), and
\[
u=\frac xn,\qquad v=\frac{x(1+\beta/s^2)}n,
\]
then, with \(r(u,v):=p(u,v,v)\),
\begin{equation}\label{eq:H-near-bound}
r(u,v)\le C\frac{x^2s}{n^2},
\qquad
H(u,v)\le C\frac{x^3}{n^3s}e^{-c\sqrt\beta}.
\end{equation}
\item If \(0<x\le\sqrt n\), \((1+\eta)x\le y\le n/2\), \(u=x/n\), and \(v=y/n\), then, with \(\ell_n\) as in Lemma \ref{lem:quantile-scale},
\begin{equation}\label{eq:H-far-bound}
H(u,v)
\le C\frac{x^3s\ell_n(x)}{n^3}\left(e^{-cs}+\left(\frac{x}{n}\right)^\rho\right)
   +C\frac{x^3s}{n^3}e^{-s^2/8}.
\end{equation}
\end{enumerate}
\end{lemma}

\begin{proof}
For part \textup{(i)}, monotonicity gives
\[
r(u,v)\le \Psi(u)=\Phi(-\sqrt2s)
\le C\frac{e^{-s^2}}{s}
\le C\frac{x^2s}{n^2},
\]
where the final bound follows from \eqref{eq:sn-exp-upper}.

We next turn \eqref{eq:8.4} into a one-dimensional bound. In part \textup{(i)},
\(v\le (1+\eta)x/n\le (1+\eta)/\sqrt n\) for all sufficiently large \(n\). In part \textup{(ii)}, the hypothesis \(y\le n/2\) gives \(v\le1/2\). Thus the denominator in \eqref{eq:8.4} is uniformly harmless on both ranges:
\((1-v)^{-2}\le C\). Since \(1-w\le1\), after writing \(z=nw\) we have
\begin{equation}\label{eq:H-proof-start}
H(u,v)\le \frac{C}{n}\int_{z_0}^{n}G\left(\frac xn,\frac zn\right)\,\dd z,
\end{equation}
where \(z_0:=nv\).

For part \textup{(i)}, split the integral in \eqref{eq:H-proof-start} at \((1+\eta)x\). In the near part write \(z=x(1+\alpha/s^2)\). Then \(\dd z=x s^{-2}\dd\alpha\), and Lemma \ref{lem:G-global-bounds}\textup{(i)} gives
\[
\frac{C}{n}\int_{\beta}^{\eta s^2}\frac{x^2s}{n^2}e^{-c_1\sqrt\alpha}\frac{x}{s^2}\,\dd\alpha
\le C\frac{x^3}{n^3s}e^{-c_2\sqrt\beta},
\]
where the polynomial factor from the tail integral in \(\alpha\) is absorbed by decreasing \(c_2\).

For the far part in \textup{(i)}, Lemma \ref{lem:G-far-integrated} and \eqref{eq:H-proof-start} give
\[
\frac{C}{n}\int_{(1+\eta)x}^{n}G\left(\frac xn,\frac zn\right)\dd z
\le
C\frac{x^3s}{n^3}
\left[\ell_n(x)\left(e^{-cs}+\left(\frac{x}{n}\right)^\rho\right)+e^{-s^2/8}\right].
\]
Since \(0\le\beta\le\eta s^2\), this far contribution is absorbed by \(C x^3 n^{-3}s^{-1}e^{-c\sqrt\beta}\) after decreasing \(c\): the first term has either exponential decay in \(s\) or a negative power of \(n\), and the last term has the Gaussian tail factor \(e^{-s^2/8}\). This proves part \textup{(i)}.

For part \textup{(ii)}, \eqref{eq:H-proof-start} starts at \(z_0=y\). Since \(y\ge(1+\eta)x\), enlarging the domain to start at \((1+\eta)x\) and applying Lemma \ref{lem:G-far-integrated} gives \eqref{eq:H-far-bound}.
\end{proof}

The next lemma collects the two expectation bounds needed in the second-moment
and prefactor estimates below, which will be used in Lemmas \ref{lem:S3-second} and
\ref{lem:prefactor-inert}.

\begin{lemma}[Mixed bounds]\label{lem:qH-bound}
\begin{equation}\label{eq:qH-bound}
\E\bigl[\Psi(U_{(1)})H(U_{(1)},U_{(2)})\bigr]
=O\left(\frac1{n^5\log n}\right),
\end{equation}
and
\begin{equation}\label{eq:H2-bound}
\E\bigl[H(U_{(1)},U_{(2)})^2\bigr]
=O\left(\frac1{n^6\log n}\right).
\end{equation}
Consequently, since \(r(u,v)\le \Psi(u)\), the same bound as in
\eqref{eq:qH-bound} holds with \(r(U_{(1)},U_{(2)})H(U_{(1)},U_{(2)})\)
in place of \(\Psi(U_{(1)})H(U_{(1)},U_{(2)})\).
\end{lemma}

\begin{proof}
Using the joint density of \((U_{(1)},U_{(2)})\), the first expectation is
\[
\E[\Psi(U_{(1)})H(U_{(1)},U_{(2)})]
=n(n-1)\int_0^1\int_u^1 \Psi(u)H(u,v)(1-v)^{n-2}\dd v\dd u.
\]
The second expectation has the same integral representation with \(H(u,v)^2\) in place of \(\Psi(u)H(u,v)\). We split these integrals according to the same near- and far-ratio regions used above. In the near region write
\[
u=\frac xn,\qquad v=\frac{x(1+\beta/s^2)}n,\qquad s=s_n(x),\qquad 0\le\beta\le\eta s^2.
\]
The joint density of \((U_{(1)},U_{(2)})\) is \(n(n-1)(1-v)^{n-2}\). With the Jacobian
\[
\dd u\,\dd v=\frac{x}{n^2s^2}\,\dd x\,\dd\beta,
\]
and the near-region bound \((1-v)^{n-2}\le Ce^{-c_0x}\), the density times the Jacobian contributes at most
\[
Cx e^{-c_0x}s^{-2}\,\dd x\,\dd\beta.
\]
Lemma \ref{lem:H-bound}\textup{(i)} and Corollary \ref{cor:goe-tail-transfer} give
\[
\Psi(u)\le C\frac{x^2s}{n^2},
\qquad
H(u,v)\le C\frac{x^3}{n^3s}e^{-c\sqrt\beta}.
\]
On the near region \(0<x\le\sqrt n\), Lemma \ref{lem:quantile-scale} gives \(s^2\ge c\log n\). Hence the near contribution to \(\E[\Psi H]\) is bounded by
\[
\frac{C}{n^5\log n}\int_0^{\sqrt n}\int_0^{\eta s^2}
x^6e^{-c_0x}e^{-c\sqrt\beta}\,\dd\beta\,\dd x
\le
\frac{C}{n^5\log n}\int_0^\infty\int_0^\infty
x^6e^{-c_0x}e^{-c\sqrt\beta}\,\dd\beta\,\dd x
=O\left(\frac1{n^5\log n}\right).
\]
Similarly, the near contribution to \(\E[H^2]\) is bounded by
\[
\frac{C}{n^6(\log n)^2}\int_0^\infty\int_0^\infty
x^7e^{-c_0x}e^{-2c\sqrt\beta}\,\dd\beta\,\dd x
=O\left(\frac1{n^6(\log n)^2}\right).
\]

It remains to check the complement, where crude estimates suffice. The region \(x>\sqrt n\) is exponentially negligible. In the far-ratio region \(0<x\le\sqrt n\), \(y\ge(1+\eta)x\), write \(u=x/n\), \(v=y/n\). The part \(y>n/2\) is exponentially small in \(n\). On \(y\le n/2\), the density contributes at most \(e^{-cy}\dd x\dd y\). Using \eqref{eq:H-far-bound} and \(\Psi(x/n)\le Cx^2s/n^2\), the contribution to \(\E[\Psi H]\) from the first term in \eqref{eq:H-far-bound} is bounded by
\[
\frac{C}{n^5}\int_0^{\sqrt n}x^5s^2\ell_n(x)e^{-cx}
\left(e^{-c_1s}+\left(\frac xn\right)^\rho\right)\dd x
=o\left(\frac1{n^5\log n}\right),
\]
and the \(e^{-s^2/8}\) term contributes
\[
\frac{C}{n^5}\int_0^{\sqrt n}x^5s^2e^{-s^2/8}e^{-cx}\,\dd x
=O\left(n^{-5-1/4}(\log n)^{9/8}\right)
=o\left(\frac1{n^5\log n}\right).
\]
The far-ratio contribution to \(\E[H^2]\) is even smaller: squaring \eqref{eq:H-far-bound} gives a factor \(n^{-6}\) times the same exponentially small or negative-power terms. Hence it is \(o(n^{-6}/\log n)\). This proves \eqref{eq:qH-bound} and \eqref{eq:H2-bound}.
\end{proof}

The remaining analysis of \(B_n\) now splits into two tasks: evaluating the linear intensity \(\E S_{3,n}\), and then using Lemma \ref{lem:qH-bound} to show that the product-linearization errors are smaller.

\begin{proposition}[Integral representation for \(\E S_{3,n}\)]\label{prop:S3-integral}
For \(n\ge4\),
\begin{equation}
\E S_{3,n}
=n(n-1)(n-2)
\int_0^1\int_u^1(1-w)
\bigl[(1-u)^{n-3}-(1-w)^{n-3}\bigr]
G(u,w)\dd w\dd u.
\label{eq:8.5}
\end{equation}
Equivalently, after \(u=x/n\), \(w=y/n\),
\begin{equation}
\E S_{3,n}
=\frac{(n-1)(n-2)}n
\int_0^n\int_x^n
\left(1-\frac yn\right)
\left[\left(1-\frac xn\right)^{n-3}-\left(1-\frac yn\right)^{n-3}\right]
G\left(\frac xn,\frac yn\right)\dd y\dd x.
\label{eq:8.6}
\end{equation}
\end{proposition}

\begin{proof}
The density of \((U_{(1)},U_{(2)})\) is
\[
f_{1,2}(u,v)=n(n-1)(1-v)^{n-2}\1_{\{0<u<v<1\}}.
\]
Using \eqref{eq:8.3} first gives
\[
\E S_{3,n}
=\frac{n(n-1)(n-2)(n-3)}2
\int_0^1\int_u^1(1-v)^{n-2}H(u,v)\dd v\dd u.
\]
Substituting \eqref{eq:8.4} cancels the factor \(1/2\) and yields
\[
\E S_{3,n}
=n(n-1)(n-2)(n-3)
\int_0^1\int_u^1(1-v)^{n-4}
\left[\int_v^1(1-w)G(u,w)\dd w\right]
\dd v\dd u.
\]
The integration domain is \(0<u<v<w<1\). Since the integrand is nonnegative, Tonelli's theorem permits the order of integration to be changed without any prior integrability assumption; see, for example, Billingsley \cite{billingsley}. Thus
\[
\E S_{3,n}
=n(n-1)(n-2)(n-3)
\int_0^1\int_u^1(1-w)G(u,w)
\left[\int_u^w(1-v)^{n-4}\dd v\right]
\dd w\dd u.
\]
The inner integral equals
\[
\int_u^w(1-v)^{n-4}\dd v
=\frac{(1-u)^{n-3}-(1-w)^{n-3}}{n-3}.
\]
After canceling \(n-3\), this is \eqref{eq:8.5}. The rescaling \(u=x/n\), \(w=y/n\) gives \eqref{eq:8.6}.
\end{proof}

\begin{theorem}[Intensity of the remaining pairs]\label{thm:S3-intensity}
\begin{equation}
\E S_{3,n}
=\left(9\sqrt{\frac\pi2}+o(1)\right)
\frac1{n(\log n)^{3/2}}.
\label{eq:8.7}
\end{equation}
\end{theorem}

\begin{proof}
Start from \eqref{eq:8.6}. The region \(x>\sqrt n\) is negligible by the same exponential estimate used in Section 7. Fix \(\eta\in(0,1]\). On \(0<x\le\sqrt n\), split the integral into the near-ratio region \(x\le y\le(1+\eta)x\) and the far-ratio region \(y\ge(1+\eta)x\). In the far-ratio region, the bracketed difference in \eqref{eq:8.6} is at most \((1-x/n)^{n-3}\le Ce^{-c_0x}\). Let \(J_{n,\mathrm{far}}\) denote the corresponding part of \eqref{eq:8.6}. Lemma \ref{lem:G-far-integrated} gives
\begin{align*}
J_{n,\mathrm{far}}
&\le Cn\int_0^{\sqrt n}e^{-c_0x}
\left[\int_{(1+\eta)x}^nG\left(\frac xn,\frac yn\right)\dd y\right]\dd x\\
&\le \frac{C}{n}\int_0^{\sqrt n}x^3s_n(x)e^{-c_0x}
\left[\ell_n(x)\left(e^{-cs_n(x)}+\left(\frac{x}{n}\right)^\rho\right)+e^{-s_n(x)^2/8}\right]\dd x.
\end{align*}
By Lemma \ref{lem:outer-integrals}, the three terms in the bracket are respectively smaller than any inverse power of \(\log n\), \(O(n^{-\rho})\), and \(O(n^{-1/4})\), up to harmless powers of \(\log n\). Hence
\[
J_{n,\mathrm{far}}=o\left(\frac1{n(\log n)^{3/2}}\right),
\]
so the far-ratio region is negligible.

In the near region set \(s=s_n(x)\) and
\begin{equation}
y=x\left(1+\frac\beta{s^2}\right),
\qquad \dd y=\frac{x}{s^2}\,\dd\beta.
\label{eq:8.8}
\end{equation}
The condition \(x\le y\le(1+\eta)x\) is exactly
\[
0\le \beta\le \eta s^2.
\]
For the leading calculation, first restrict to a compact rectangle \(x\in[\delta,M]\), \(0\le\beta\le B\). On this rectangle, \(s^2\sim2\log n\), and the mean-value theorem gives
\begin{align}
&s^2\left[
\left(1-\frac xn\right)^{n-3}
-
\left(1-\frac{x(1+\beta/s^2)}n\right)^{n-3}
\right] \notag\\
&\qquad
=s^2\frac{x\beta}{ns^2}(n-3)
\left(1-\frac{\xi_n}{n}\right)^{n-4}
\longrightarrow x\beta e^{-x},
\label{eq:8.9}
\end{align}
where \(\xi_n\in[x,x(1+\beta/s^2)]\). Together with Lemma \ref{lem:G-boundary-layer} and the Jacobian \(\dd y=x s^{-2}\,\dd\beta\), this gives the pointwise limit of the rescaled integrand in \eqref{eq:8.6}: after multiplying the full integrand by \(n(\log n)^{3/2}\), it converges to
\[
\frac{\sqrt\pi}{2\sqrt2}x^4\beta e^{-x}e^{-2\sqrt\beta}.
\]

It remains to justify extending the compact rectangle to the full near-ratio region. We record the elementary estimate used for this purpose. Since
\[
\left(1-\frac xn\right)^{n-3}
-
\left(1-\frac yn\right)^{n-3}
=\frac{n-3}{n}\int_x^y\left(1-\frac rn\right)^{n-4}\dd r,
\]
and \(x\le r\le y\le(1+\eta)x\le2\sqrt n\), the integrand is bounded by \(Ce^{-c_0x}\). Since the integration interval has length
\[
y-x=\frac{x\beta}{s^2},
\]
there are constants \(C,c_0>0\) such that, for all large \(n\),
\begin{equation}
0\le
\left(1-\frac xn\right)^{n-3}
-
\left(1-\frac{x(1+\beta/s^2)}n\right)^{n-3}
\le C e^{-c_0x}\frac{x\beta}{s^2}.
\label{eq:8.10}
\end{equation}
Combining \eqref{eq:8.10} with Lemma \ref{lem:G-global-bounds}(i), the Jacobian \(\dd y=x s^{-2}\,\dd\beta\), and the boundedness of \((\log n)^{3/2}/s^3\) gives an \(n\)-uniform bound for the rescaled integrand:
\[
C x^4\beta e^{-c_0x}e^{-c\sqrt\beta}.
\]
This is an integrable dominating function on \((0,\infty)^2\). The extension from compact rectangles is therefore rigorous by the same truncate-then-dominated-convergence argument used in Section \ref{sec:A}: choose \(\delta,M,B\) so that the tail integral of this bound is arbitrarily small, then apply dominated convergence on \([\delta,M]\times[0,B]\). The growing upper limit \(\eta s^2\) in \(\beta\) is harmless because the tail \(\beta>B\) is controlled by \(e^{-c\sqrt\beta}\). Consequently,
\[
\lim_{n\to\infty} n(\log n)^{3/2}\E S_{3,n}
=\frac{\sqrt\pi}{2\sqrt2}
\int_0^\infty x^4e^{-x}\dd x
\int_0^\infty \beta e^{-2\sqrt\beta}\dd\beta.
\]
Now
\[
\int_0^\infty x^4e^{-x}\dd x=24,
\qquad
\int_0^\infty\beta e^{-2\sqrt\beta}\dd\beta=\frac34.
\]
The product is \(9\sqrt{\pi/2}\), proving \eqref{eq:8.7}.
\end{proof}

\subsection{Reduction of \texorpdfstring{\(B_n\)}{Bn} to its linear intensity}

Set
\begin{equation}
W_n:=(1-q_n)^{n-1}\prod_{j=3}^n(1-p_{2j}),
\label{eq:8.11}
\end{equation}
so that
\begin{equation}
B_n=\E\left[W_n\left(1-\prod_{3\le i<j\le n}(1-p_{ij})\right)\right].
\label{eq:8.12}
\end{equation}

\begin{lemma}[Deterministic linearization]\label{lem:B-linearization}
Let
\[
T_{3,n}:=1-\prod_{3\le i<j\le n}(1-p_{ij}).
\]
Then
\begin{equation}
S_{3,n}-\frac12S_{3,n}^2\le T_{3,n}\le S_{3,n}.
\label{eq:8.13}
\end{equation}
Consequently,
\begin{equation}
\E S_{3,n}-\frac12\E S_{3,n}^2-\E[(1-W_n)S_{3,n}]
\le B_n\le \E S_{3,n}.
\label{eq:8.14}
\end{equation}
\end{lemma}

\begin{proof}
The deterministic inequality \eqref{eq:8.13} is Lemma \ref{lem:product-linearization} applied to the list \(\{p_{ij}:3\le i<j\le n\}\). Since \(0\le W_n\le1\), multiplying \eqref{eq:8.13} by \(W_n\), taking expectations, and using
\[
\E[W_nS_{3,n}]=\E S_{3,n}-\E[(1-W_n)S_{3,n}]
\]
gives \eqref{eq:8.14}.
\end{proof}

It remains to bound the two error terms.

\begin{lemma}[Negligible second moment]\label{lem:S3-second}
\begin{equation}
\E S_{3,n}^2=O\left(\frac1{n^2\log n}\right)=o(\E S_{3,n}).
\label{eq:8.15}
\end{equation}
\end{lemma}

\begin{proof}
Condition on \((U_{(1)},U_{(2)})=(u,v)\), and let \(m=n-2\). Let \(X_1,\ldots,X_m\) be i.i.d. uniform on \((v,1)\). Define
\[
h_u(x,y):=p(u,\min\{x,y\},\max\{x,y\}).
\]
Under this conditional law, \(S_{3,n}\) has the same distribution as the pair sum
\[
\sum_{1\le a<b\le m}h_u(X_a,X_b).
\]
Let \(r(u,v):=p(u,v,v)\). Monotonicity of the pair probability gives \(0\le h_u(x,y)\le r(u,v)\). Since \(X_1,X_2\) are independent uniform variables on \((v,1)\), the definition of \(H\) gives
\[
\E[h_u(X_1,X_2)]=H(u,v).
\]
Now expand the square of this pair sum. There are three types of terms. Identical pairs, of which there are \(\binom m2\le m^2\), contribute at most
\[
\binom m2\E h_u(X_1,X_2)^2
\le C m^2 r(u,v)H(u,v).
\]
Pairs sharing exactly one index, of which there are \(3\binom m3\le m^3\), may be represented by \((1,2)\) and \((1,3)\). Their contribution is at most
\[
C m^3\E[h_u(X_1,X_2)h_u(X_1,X_3)]
\le C m^3 r(u,v)H(u,v),
\]
because one factor is bounded by \(r(u,v)\) almost surely, so
\[
\E[h_u(X_1,X_2)h_u(X_1,X_3)]
\le \E[r(u,v)h_u(X_1,X_3)]=r(u,v)H(u,v).
\]
Disjoint pairs, such as \((1,2)\) and \((3,4)\), are independent and contribute at most
\[
C m^4 H(u,v)^2.
\]
Absorbing the identical-pair term into the shared-index term gives
\begin{equation}
\E(S_{3,n}^2\mid u,v)
\le C\left(m^3r(u,v)H(u,v)+m^4H(u,v)^2\right).
\label{eq:8.16}
\end{equation}
Taking expectations in \eqref{eq:8.16} and using \(m\le n\), \(r\le \Psi\), and Lemma \ref{lem:qH-bound}, we get
\[
\E S_{3,n}^2
\le C\left(n^3\E[\Psi(U_{(1)})H(U_{(1)},U_{(2)})]
+n^4\E[H(U_{(1)},U_{(2)})^2]\right)
=O\left(\frac1{n^2\log n}\right).
\]
This proves \eqref{eq:8.15}.
\end{proof}

\begin{lemma}[The prefactor is asymptotically inert]\label{lem:prefactor-inert}
\begin{equation}
\E[(1-W_n)S_{3,n}]
=O\left(\frac1{n^2\log n}\right)=o(\E S_{3,n}).
\label{eq:8.18}
\end{equation}
\end{lemma}

\begin{proof}
Applying Lemma \ref{lem:product-linearization} in the form \(1-\prod_r(1-a_r)\le\sum_r a_r\) to the factors defining \(W_n\),
\begin{equation}
1-W_n\le (n-1)q_n+\sum_{j=3}^n p_{2j}.
\label{eq:8.19}
\end{equation}
Therefore
\begin{equation}
\E[(1-W_n)S_{3,n}]
\le \E[(n-1)q_nS_{3,n}]
+\E\left[S_{3,n}\sum_{j=3}^n p_{2j}\right].
\label{eq:8.20}
\end{equation}
Conditioning on \((U_{(1)},U_{(2)})=(u,v)\), the first term is controlled by
\[
(n-1)\Psi(u)\binom{n-2}{2}H(u,v).
\]
For the second term, each summand \(p_{2j}\) is at most \(r(u,v):=p(u,v,v)\), and there are \(n-2\) such summands. Hence its conditional contribution is at most
\[
(n-2)r(u,v)\binom{n-2}{2}H(u,v).
\]
Since \(r(u,v)\le \Psi(u)\), this has the same order as the first conditional contribution, up to an absolute constant.

Both terms in \eqref{eq:8.20} are now controlled by Lemma \ref{lem:qH-bound}. Indeed, the first conditional contribution is at most \(Cn^3\Psi(u)H(u,v)\). For the second contribution, \(r(u,v)\le \Psi(u)\), so it is also at most \(Cn^3\Psi(u)H(u,v)\). Taking expectations and applying Lemma \ref{lem:qH-bound} gives, for each term in \eqref{eq:8.20},
\[
Cn^3\E\bigl[\Psi(U_{(1)})H(U_{(1)},U_{(2)})\bigr]
=O\left(\frac1{n^2\log n}\right).
\]
Therefore both contributions in \eqref{eq:8.20} are bounded by
\[
O\left(n^{-2}(\log n)^{-1}\right).
\]
\end{proof}

\begin{theorem}[Asymptotics of \(B_n\)]\label{thm:B-asympt}
\begin{equation}
B_n=\left(9\sqrt{\frac\pi2}+o(1)\right)
\frac1{n(\log n)^{3/2}}.
\label{eq:8.21}
\end{equation}
\end{theorem}

\begin{proof}
Lemma \ref{lem:B-linearization} gives
\[
\E S_{3,n}-\frac12\E S_{3,n}^2-\E[(1-W_n)S_{3,n}]
\le B_n\le\E S_{3,n}.
\]
By Lemmas \ref{lem:S3-second} and \ref{lem:prefactor-inert}, both error terms are \(o(\E S_{3,n})\). The result follows from Theorem \ref{thm:S3-intensity}.
\end{proof}

\section{Tail bound for support sizes at least three}\label{sec:support-three}

This section proves that supports of size at least three are negligible at the scale of \(B_n\):
\begin{equation}
\Prob(\kappa_n\ge3)=o\left(\frac1{n(\log n)^{3/2}}\right).
\label{eq:9.1}
\end{equation}
The proof has two parts.  First we recall the fixed-support bound of Chen and Peng \cite{chen-peng}, which is the quantitative estimate behind the qualitative statement \(\Prob(\kappa_n\ge3)\to0\).  Then we sharpen only the support-three term, because the Chen--Peng bound is already sufficient for all \(k\ge4\).

\subsection{Historical fixed-support bounds}

Chen and Peng \cite[Theorem 9]{chen-peng}, using both the first-order and second-order optimality conditions, proved the following GOE fixed-support estimate.  We state it in the present normalization.

\begin{theorem}[Chen--Peng fixed-support tail bound]\label{thm:chen-peng-fixed}
There exists \(\eta>0\) such that for every \(k\ge2\),
\begin{equation}
\Prob(\kappa_n=k)
\le
\frac{(2k-3)!}{(k-1)!(n+1)^{k-2}}
\left(\eta^2\log n+\frac{k-1}{2k-2}\right)^{k-1}
\exp\left(-\frac{(k-2)^2}{4}\right).
\label{eq:9.2}
\end{equation}
\end{theorem}

At \(k=3\), Theorem \ref{thm:chen-peng-fixed} gives the explicit bound
\begin{equation}
\Prob(\kappa_n=3)
\le
\frac{3}{n+1}\left(\eta^2\log n+\frac12\right)^2e^{-1/4}
=O\left(\frac{(\log n)^2}{n}\right).
\label{eq:9.3}
\end{equation}
Together with the summation over \(k\ge4\) below, this gives the Chen--Peng-based estimate
\[
\Prob(\kappa_n\ge3)=O\left(\frac{(\log n)^2}{n}\right),
\]
which is the bound used by Bomze, Schachinger, and Ullrich \cite{bomze-schachinger-ullrich} in their discussion of typical random StQP instances. However, this rate is not sharp enough for the present expansion, because
\[
\frac{(\log n)^2/n}{1/(n(\log n)^{3/2})}=(\log n)^{7/2}\to\infty.
\]
Thus a sharper treatment of the support-three term is needed.  Chen and Pittel \cite{chen-pittel} later obtained polylogarithmic support-size bounds under much broader one-tail assumptions; for the exact GOE case considered here, their results broaden distributional scope rather than improving the fixed-\(k=3\) rate in \eqref{eq:9.3}.

For \(k\ge4\), the Chen--Peng estimate is sufficient.

\begin{lemma}[Tail for \(k\ge4\)]\label{lem:kge4-tail}
\begin{equation}
\Prob(\kappa_n\ge4)
\le C\frac{(\log n)^3}{n^2}.
\label{eq:9.4}
\end{equation}
\end{lemma}

\begin{proof}
From \eqref{eq:9.2}, for \(k\ge4\),
\[
\Prob(\kappa_n=k)
\le
\frac{(\log(n+1))^3}{(n+1)^2}\,\beta_k r_n^{k-4},
\]
where
\[
\beta_k:=\frac{(2k-3)!}{(k-1)!}\eta^{2k-2}
\exp\left(-\frac{(k-2)^2}{4}\right),
\qquad
r_n:=\frac{\log(n+1)}{n+1}.
\]
For large \(n\), \(r_n\le1/2\), and the Gaussian factor in \(\beta_k\) dominates the factorial growth. Hence
\[
\sum_{k\ge4}\beta_k2^{-(k-4)}<\infty.
\]
Summing over \(k\ge4\) proves \eqref{eq:9.4}.
\end{proof}

\subsection{A sharper row-average estimate}

The improvement over \eqref{eq:9.3} comes from retaining the active diagonal term in the row-average necessary condition. For a fixed support \(K\) of size \(k\) and a row \(i\in K\), define
\begin{equation}
R_{K,i}:=\frac1k\sum_{j\in K}Q_{ij}.
\label{eq:9.5}
\end{equation}
Then
\[
R_{K,i}\sim N\left(0,\sigma_k^2\right),
\qquad
\sigma_k^2:=\frac{k+1}{2k^2},
\]
because the sum contains one diagonal term of variance \(1\) and \(k-1\) off-diagonal terms of variance \(1/2\). Let \(M_m\) denote the minimum of \(m\) independent \(N(0,1)\) random variables.

\begin{theorem}[Support-three and fixed-\(k\) row-average bounds]\label{thm:support-three-bound}
For every fixed \(k\ge3\), define
\[
\alpha_k:=\sigma_k^{-2}=\frac{2k^2}{k+1}=2k-2+\frac{2}{k+1}.
\]
There exists \(C_k<\infty\) such that
\begin{equation}
\Prob(\kappa_n=k)
\le C_k n^{2-k-2/(k+1)}(\log n)^{k-3/2+1/(k+1)}.
\label{eq:9.10}
\end{equation}
In particular, for \(k=3\), there exists \(C>0\) such that
\begin{equation}
\Prob(\kappa_n=3)
\le C\frac{(\log n)^{7/4}}{n^{3/2}}.
\label{eq:9.7}
\end{equation}
Consequently,
\begin{equation}
\Prob(\kappa_n\ge3)
\le C\frac{(\log n)^{7/4}}{n^{3/2}}
=o\left(\frac1{n(\log n)^{3/2}}\right).
\label{eq:9.8}
\end{equation}
\end{theorem}

\begin{proof}
Fix \(k\ge3\), a support \(K\) of size \(k\), and a row \(i\in K\). If \(K\) is the support of the global optimizer and row \(i\) is selected by Proposition \ref{prop:det-support}, then
\[
R_{K,i}
<\min_{1\le r\le n}Q_{rr}
\le \min_{r\notin K}Q_{rr}.
\]
The random variable \(R_{K,i}\) depends only on the entries in row \(i\) restricted to \(K\), while \(\min_{r\notin K}Q_{rr}\) depends only on the diagonal entries outside \(K\). Hence the two variables are independent. If \(Y_k\sim N(0,\sigma_k^2)\) is independent of \(M_{n-k}\), then
\[
\Prob\left(R_{K,i}<\min_{r\notin K}Q_{rr}\right)
=
\Prob\left(Y_k<M_{n-k}\right).
\]
Applying Theorem \ref{thm:normal-minimum} with \(a=1/\sigma_k=\sqrt{\alpha_k}\), and then taking a union bound over the \(k\binom nk\) choices of \((K,i)\), gives
\[
\Prob(\kappa_n=k)
\le C_k n^k n^{-\alpha_k}(\log n)^{(\alpha_k-1)/2}.
\]
Since \(k-\alpha_k=2-k-2/(k+1)\) and \((\alpha_k-1)/2=k-3/2+1/(k+1)\), this proves \eqref{eq:9.10}.

For \(k=3\), \(\sigma_3^2=2/9\) and \(\alpha_3=9/2\). Substituting these values in \eqref{eq:9.10} gives \eqref{eq:9.7}. Adding the bound \eqref{eq:9.4} for \(k\ge4\) gives \eqref{eq:9.8}.
\end{proof}

Compared with the historical bound \eqref{eq:9.3}, the support-three estimate in Theorem \ref{thm:support-three-bound} gains the factor
\[
\frac{(\log n)^2/n}{(\log n)^{7/4}/n^{3/2}}
=\sqrt n\,(\log n)^{1/4}.
\]
This gain is precisely what is needed to make the support-three contribution negligible relative to \(B_n\).

\begin{remark}[Variance improvement, second-order conditions, and fixed-\(k\) comparison]
The Chen--Peng estimate \eqref{eq:9.2} applies across general support sizes and is derived from a fixed-support analysis using both first- and second-order optimality conditions. If \(k\ge3\) is fixed, \eqref{eq:9.2} has order
\[
n^{-(k-2)}(\log n)^{k-1},
\]
whereas the fixed-\(k\) bound in Theorem \ref{thm:support-three-bound} gives
\[
n^{-(k-2)-2/(k+1)}(\log n)^{k-3/2+1/(k+1)}.
\]
Thus Theorem \ref{thm:support-three-bound} gives a sharper fixed-\(k\) upper-bound order, using only the row-average necessary condition \eqref{eq:row-avg-cond}.

The source of the improvement is the active diagonal term. For a fixed support \(K\) of size \(k\), the exact row average
\[
\frac1k\sum_{j\in K}Q_{ij}
\]
has variance \((1+(k-1)/2)/k^2=(k+1)/(2k^2)\). By contrast, the off-diagonal-only surrogate
\[
\frac1{k-1}\sum_{j\in K\setminus\{i\}}Q_{ij}
\]
has variance \((k-1)(1/2)/(k-1)^2=1/(2k-2)\), which is larger. Chen and Peng drop the active diagonal entry in their one-row relaxation to avoid a dependence issue when the row condition is combined with the second-order optimality condition on the active principal submatrix. The price of that decoupling is the larger variance of the off-diagonal-only surrogate. For \(k=3\), retaining the diagonal term improves the rare-event exponent from \(4\) to \(9/2\), producing the extra factor \(n^{-1/2}\) in Theorem \ref{thm:support-three-bound}.

Although the bound in Theorem \ref{thm:support-three-bound} uses only the row-average condition, one may try to combine this diagonal-retaining row-average estimate with the second-order optimality condition. We do not pursue this refinement, because the support-three estimate is already negligible on the \(B_n\) scale. A sharp fixed-\(k\) asymptotic for \(k\ge3\), using both first- and second-order conditions, remains a separate problem.
\end{remark}

\section{Main theorem}

We now combine the product decomposition with the estimates for \(S_n\), \(A_n\), \(B_n\), and the support-size tail.

Recall
\[
S_n:=\E\left[1-(1-q_n)^{n-1}\right],
\qquad q_n=\Phi(\sqrt2 Z_{(1)}).
\]

\begin{theorem}[Expansion for singleton optimality]\label{thm:singleton-expansion}
For the GOE standard quadratic program \eqref{eq:1.1},
\begin{equation}
\Prob(\kappa_n=2)=S_n+A_n+B_n+o(B_n),
\label{eq:10.1}
\end{equation}
\begin{equation}
\Prob(\kappa_n>1)=S_n+A_n+B_n+o(B_n),
\label{eq:10.2}
\end{equation}
and
\begin{equation}
\Prob(\kappa_n=1)=1-S_n-A_n-B_n+o(B_n),
\label{eq:10.3}
\end{equation}
where
\begin{equation}
A_n\sim3\sqrt{\frac\pi2}\frac1{n\sqrt{\log n}},
\qquad
B_n\sim9\sqrt{\frac\pi2}\frac1{n(\log n)^{3/2}}.
\label{eq:10.4}
\end{equation}
Moreover,
\begin{equation}
S_n\sim2\sqrt{2\pi}\frac{\sqrt{\log n}}n.
\label{eq:10.5}
\end{equation}
Consequently,
\begin{equation}
\Prob(\kappa_n>1)
\sim2\sqrt{2\pi}\frac{\sqrt{\log n}}n,
\qquad
\Prob(\kappa_n=1)=1-\Theta\left(\frac{\sqrt{\log n}}n\right).
\label{eq:10.6}
\end{equation}
\end{theorem}

\begin{proof}
Proposition \ref{prop:product-decomp} gives
\[
\Prob(\kappa_n=2)=S_n+A_n+B_n+O(\Prob(\kappa_n\ge3)).
\]
By Theorem \ref{thm:support-three-bound},
\[
\Prob(\kappa_n\ge3)=o(B_n).
\]
This proves \eqref{eq:10.1}. Since
\[
\Prob(\kappa_n>1)=\Prob(\kappa_n=2)+\Prob(\kappa_n\ge3),
\]
\eqref{eq:10.2} follows. Equation \eqref{eq:10.3} is the complement. The asymptotics \eqref{eq:10.4} are Corollary \ref{cor:A-asympt} and Theorem \ref{thm:B-asympt}, while \eqref{eq:10.5} is Proposition \ref{prop:S-asympt}. Finally, \(S_n\) dominates both \(A_n\) and \(B_n\), proving \eqref{eq:10.6}.
\end{proof}

\begin{remark}[First algebraic correction to the minimum-diagonal contribution]
The term \(S_n\) admits the refinement
\begin{equation}
S_n=\frac{2\sqrt{2\pi}\sqrt{L_n}}{n}
-\sqrt{\frac\pi2}\,\frac{\log L_n+\log(4\pi)-2\gamma_E}{n\sqrt{L_n}}
+o\left(\frac1{n\sqrt{L_n}}\right),
\qquad L_n:=\log n,
\label{eq:10.7}
\end{equation}
where \(\gamma_E\) is Euler's constant; see, for example, Abramowitz and Stegun \cite[Section 6.1]{abramowitz-stegun}. This follows by writing \(S_n=(n-1)\E q_n+o((n\sqrt{L_n})^{-1})\), using
\[
\Phi(\sqrt2\Phi^{-1}(u))=\sqrt\pi\,u^2\left(s+\frac{3}{2s}+O(s^{-3})\right),
\qquad s=-\Phi^{-1}(u).
\]
The term \(3/(2s)\) comes directly from dividing the second-order Mills expansions
\[
u^2=\frac{e^{-s^2}}{2\pi s^2}\left(1-2s^{-2}+O(s^{-4})\right),
\qquad
\Phi(-\sqrt2s)=\frac{e^{-s^2}}{2\sqrt\pi s}\left(1-\frac12s^{-2}+O(s^{-4})\right).
\]
Their ratio is
\[
\frac{\Phi(-\sqrt2s)}{u^2}
=\sqrt\pi\,s\,\frac{1-\frac12s^{-2}+O(s^{-4})}{1-2s^{-2}+O(s^{-4})}
=\sqrt\pi\left(s+\frac{3}{2s}+O(s^{-3})\right),
\]
which gives the displayed expansion for \(\Psi(u)=\Phi(-\sqrt2s)\).
We then insert the standard lower-tail quantile expansion \cite[Section 2.2]{leadbetter}
\[
s=\sqrt{2\log(1/u)}-\frac{\log\log(1/u)+\log(4\pi)}{2\sqrt{2\log(1/u)}}+O\left(\frac{(\log\log(1/u))^2}{(\log(1/u))^{3/2}}\right).
\]
After the change of variables \(u=t/n\), the logarithmic moment is
\[
\int_0^\infty t^2e^{-t}\log t\,\dd t=\Gamma'(3)=\Gamma(3)\psi(3)=3-2\gamma_E,
\]
where \(\psi\) is the digamma function. This identity is the source of the Euler-constant term in \eqref{eq:10.7}; combined with the second-order Mills term, it yields the displayed correction. This \((\log\log n)/(n\sqrt{\log n})\) contribution is the reason the minimum-diagonal term is left exact in Theorem \ref{thm:singleton-expansion}.
\end{remark}

\begin{corollary}[Leading non-vertex and support-two constants]
The scale \(\sqrt{\log n}/n\) is the exact scale of the probability that the optimizer is not a vertex, and the support-two probability has the same leading constant:
\[
\lim_{n\to\infty}\frac{n}{\sqrt{\log n}}\Prob(\kappa_n>1)
=
\lim_{n\to\infty}\frac{n}{\sqrt{\log n}}\Prob(\kappa_n=2)
=2\sqrt{2\pi}.
\]
\end{corollary}

\section{Conclusion and extensions}

The finite-sample inequality framework and the boundary-layer calculation give complementary information about sparsity in GOE standard quadratic programs. The quantile-sandwich and beta-log estimates explain the finite-\(n\) scale of one-row KKT necessary conditions, while the exact two-point product decomposition identifies the leading constant. Theorem \ref{thm:singleton-expansion} also has a simple algorithmic implication: for independent GOE data, choosing the coordinate with the minimum diagonal entry solves almost every large instance. Thus independent GOE matrices should not be regarded as generic hard benchmarks for nonconvex StQP algorithms. This is consistent with the direction pursued by Bomze, Schachinger, and Ullrich \cite{bomze-schachinger-ullrich}, who, motivated by Chen, Peng, and Zhang \cite{chen-peng-zhang} and Chen and Peng \cite{chen-peng}, construct and analyze StQP instances with rich solution structures.

The proof separates deterministic two-coordinate geometry from probabilistic input. The two-point threshold is deterministic and therefore holds for any symmetric matrix. For generalized Wigner ensembles with independent diagonal and off-diagonal entries, the conditional product structure remains available, but the leading constant becomes a distribution-dependent extreme-value quantity involving the lower diagonal tail and the off-diagonal distribution. For sample-covariance or positive-semidefinite models, the threshold still holds, but entrywise dependence prevents the diagonal-conditioned product representation; see Bai--Silverstein \cite{bai-silverstein} for background on sample-covariance matrices. Band matrices and spatially dependent Wigner models would require local versions of the pair corrections; see Bourgade, Erdos, Yau, and Yin \cite{bourgade} for related random band matrix results. Heavy-tailed ensembles would replace the Gaussian quantile calculations by heavy-tail extreme-value estimates; see Ben Arous--Guionnet \cite{benarous-guionnet}.

Within the GOE model, natural refinements include the ordinary asymptotic expansion of the exact term \(S_n\), matching asymptotics for \(\Prob(\kappa_n=k)\) at fixed \(k\ge3\), anisotropic Gaussian models with non-identical variances, and universality results identifying the weakest entrywise assumptions under which the GOE constants are replaced by explicit distribution-dependent analogues.

\appendix

\section{Order-indexed one-row relaxation probabilities}\label{app:one-row}

This appendix records probabilities indexed by the diagonal order index for a one-row KKT relaxation. Here an \emph{order index} is the integer \(k\) in the ordered diagonal variable \(Z_{(k)}\). The derived probability asymptotics improve the relaxation-level estimates implicit in Chen--Peng \cite{chen-peng} and may be of independent interest. Their main role in this paper is interpretive: they explain how a one-row KKT calculation is related to \(\Prob(\kappa_n=2)\), why that calculation gives the correct order \(\sqrt{\log n}/n\), and why it does not give the exact constant.

\subsection{Relation with the support-two probability}

For \(1\le k<j\le n\), define the one-row event
\begin{equation}\label{eq:row-relaxed-pair-event}
E_{k,j}:=\{X_{k j}<2Z_{(1)}-Z_{(k)}\}.
\end{equation}
This is the row-average test
\[
\frac{Z_{(k)}+X_{k j}}{2}<Z_{(1)}.
\]
For fixed \(k\), let \(P_{n,k}\) denote the unconditional probability of this one-row event for any one ordered pair with first endpoint of order index \(k\): if \(X\sim N(0,1/2)\) is independent of the diagonal order statistics, set
\begin{equation}\label{eq:Pnk-def}
P_{n,k}:=\Prob\{X<2Z_{(1)}-Z_{(k)}\}.
\end{equation}

The exact two-coordinate event is \(F_{k j}=\{X_{k j}<\tau_{k j}\}\), with \(\tau_{k j}\) defined in \eqref{eq:4.4}. Since
\[
\tau_{k j}
=Z_{(1)}-\sqrt{(Z_{(k)}-Z_{(1)})(Z_{(j)}-Z_{(1)})}
\le 2Z_{(1)}-Z_{(k)},
\]
we have the deterministic inclusion
\begin{equation}\label{eq:F-row-inclusion}
F_{k j}\subseteq E_{k,j},\qquad 1\le k<j\le n.
\end{equation}
Therefore, using \eqref{eq:4.9},
\begin{equation}\label{eq:k2-Pnk-bound}
\Prob(\kappa_n=2)
\le
\Prob\left(\bigcup_{1\le k<j\le n}F_{k j}\right)
\le
\Prob\left(\bigcup_{1\le k<j\le n}E_{k,j}\right).
\end{equation}
For fixed \(k\), the events \(E_{k,j}\), \(j>k\), have probability \(P_{n,k}\). Hence the union bound gives the first-moment bound
\begin{equation}\label{eq:Pnk-first-moment-bound}
\Prob(\kappa_n=2)
\le
\sum_{k=1}^{n-1}(n-k)P_{n,k}.
\end{equation}
Thus \(P_{n,k}\) should be read as the contribution of a fixed ordered row index \(k\) to a one-row first-moment bound for the support-two probability. The inequalities above lose information in two ways: the threshold in \eqref{eq:row-relaxed-pair-event} is weaker than the exact curved threshold, and the final step is a union bound over many dependent candidate pairs. The main text avoids both losses by using the exact threshold and the conditional product formula.

The case \(k=1\) is special. Since \(\tau_{1j}=Z_{(1)}\), the one-row event \(E_{1,j}\) coincides with the exact event \(F_{1j}\). Hence the first order index already produces the correct leading scale, while the remaining order indices explain the size of the one-row upper bound.

\subsection{Fixed order indices}

For \(k=1\), Theorem \ref{thm:normal-minimum} with \(a=\sqrt2\) gives
\begin{equation}\label{eq:Pn1-asymp}
P_{n,1}
\sim 2\sqrt{2\pi}\,\frac{\sqrt{\log n}}{n^2}.
\end{equation}
Consequently, the first order index contributes
\[
(n-1)P_{n,1}
\sim 2\sqrt{2\pi}\,\frac{\sqrt{\log n}}{n}
\]
to the first-moment bound \eqref{eq:Pnk-first-moment-bound}. This is the same order as \(\Prob(\kappa_n=2)\) in Theorem \ref{thm:singleton-expansion}.

For \(k\ge2\), write \(U_{(1)}:=\Phi(Z_{(1)})\) and \(U_{(k)}:=\Phi(Z_{(k)})\).  The joint density of \((U_{(1)},U_{(k)})\) is
\begin{equation}\label{eq:app-u1-uk-density}
\frac{n!}{(k-2)!(n-k)!}(v-u)^{k-2}(1-v)^{n-k},
\qquad 0<u<v<1.
\end{equation}
Conditioning on \((U_{(1)},U_{(k)})=(u,v)\) gives
\begin{equation}\label{eq:Pnk-integral}
\begin{aligned}
P_{n,k}
&=\frac{n!}{(k-2)!(n-k)!}
\int_0^1\int_u^1
\Phi\!\left(\sqrt2\{2\Phi^{-1}(u)-\Phi^{-1}(v)\}\right) \\
&\hspace{35mm}\times (v-u)^{k-2}(1-v)^{n-k}\,\dd v\,\dd u .
\end{aligned}
\end{equation}

\begin{proposition}[Fixed-order-index one-row asymptotics]\label{prop:app-Pnk-fixed}
For every fixed integer \(k\ge2\),
\begin{equation}\label{eq:Pnk-asymp}
P_{n,k}\sim
\frac{24\sqrt{2\pi}}{(k+2)(k+3)}\,\frac{\sqrt{\log n}}{n^2}.
\end{equation}
\end{proposition}

\begin{proof}
Set \(u=x/n\), \(v=y/n\), and first fix \(0<x<y<\infty\).  Let
\[
s=s_n(x),\qquad t=s_n(y),\qquad h=s-t.
\]
Since \(\Phi^{-1}(x/n)=-s\) and \(\Phi^{-1}(y/n)=-t\),
\begin{equation}\label{eq:app-threshold-sth}
2\Phi^{-1}(u)-\Phi^{-1}(v)=-(s+h).
\end{equation}
The fixed-\(x,y\) version of the Mills expansion used in the proof of Lemma \ref{lem:quantile-gap}, together with Lemma \ref{lem:quantile-scale}, gives
\begin{equation}\label{eq:app-fixed-rank-inputs}
sh\to\log(y/x),\qquad h^2\to0,
\qquad
\frac{e^{-s^2}}{s}
=\frac{2\pi x^2\sqrt{2\log n}}{n^2}(1+o(1)).
\end{equation}
Mills' bounds applied to \(\Phi(-\sqrt2(s+h))\), using \eqref{eq:app-threshold-sth} and \eqref{eq:app-fixed-rank-inputs}, yield
\begin{equation}\label{eq:Pnk-integrand-limit}
\Phi\!\left(\sqrt2\{2\Phi^{-1}(u)-\Phi^{-1}(v)\}\right)
=\sqrt{2\pi}\,\frac{x^4}{y^2}\frac{\sqrt{\log n}}{n^2}(1+o(1)).
\end{equation}
After the same change of variables, the density factor in \eqref{eq:Pnk-integral} satisfies
\begin{equation}\label{eq:app-order-density-limit}
\frac{n!}{(k-2)!(n-k)!}(v-u)^{k-2}(1-v)^{n-k}\,\dd v\,\dd u
\to
\frac{(y-x)^{k-2}e^{-y}}{(k-2)!}\,\dd y\,\dd x.
\end{equation}
The convergence is dominated by the tail-transfer and quantile-gap bounds in Section \ref{sec:toolbox}.  For bounded \(y\), the displayed limits are uniform on compact subsets of \(0<x<y\); for large \(y\), the factor \((1-v)^{n-k}\le e^{-ny/2n}=e^{-y/2}\) when \(v=y/n\le1/2\) gives exponential decay; and the contribution of \(v>1/2\) is exponentially small because then \((1-v)^{n-k}\le2^{-(n-k)}\).  Hence dominated convergence gives
\begin{align}
P_{n,k}
&\sim
\sqrt{2\pi}\,\frac{\sqrt{\log n}}{n^2}
\frac1{(k-2)!}
\int_0^\infty\int_x^\infty
\frac{x^4}{y^2}(y-x)^{k-2}e^{-y}\,\dd y\,\dd x.
\label{eq:app-Pnk-double-integral}
\end{align}
It remains to evaluate the integral.  Set \(x=zy\), with \(0<z<1\).  Since \(\dd x=y\,\dd z\),
\begin{align}
\int_0^\infty\int_x^\infty
\frac{x^4}{y^2}(y-x)^{k-2}e^{-y}\,\dd y\,\dd x
&=\int_0^\infty y^{k+1}e^{-y}\,\dd y
\int_0^1 z^4(1-z)^{k-2}\,\dd z \notag\\
&=\Gamma(k+2)\frac{\Gamma(5)\Gamma(k-1)}{\Gamma(k+4)} \notag\\
&=\frac{24(k-2)!}{(k+2)(k+3)}.
\label{eq:app-Pnk-beta-eval}
\end{align}
Substituting \eqref{eq:app-Pnk-beta-eval} into \eqref{eq:app-Pnk-double-integral} proves \eqref{eq:Pnk-asymp}.
\end{proof}

\subsection{Bulk order indices and the scale of the first-moment bound}

The fixed-order-index calculation identifies the boundary contribution to the first-moment bound \eqref{eq:Pnk-first-moment-bound}. We next explain why order indices far from the minimum do not alter that scale.

Let \(m_n=\lfloor n^\alpha\rfloor\), where \(0<\alpha<1\), and put \(L=\log n\).  The normal quantile approximation gives
\begin{equation}\label{eq:app-poly-rank-quantiles}
Z_{(1)}=-\sqrt{2L}+o(\sqrt L),
\qquad
Z_{(m_n)}=-\sqrt{2(1-\alpha)L}+o(\sqrt L)
\end{equation}
in probability. The following calculation is an order-level comparison included only to explain why bulk order indices do not change the boundary scale. It gives
\begin{equation}\label{eq:app-poly-rank-threshold}
2Z_{(1)}-Z_{(m_n)}
=-\sqrt{2L}\bigl(2-\sqrt{1-\alpha}\bigr)+o(\sqrt L).
\end{equation}
Set
\[
a_\alpha:=2-\sqrt{1-\alpha}.
\]
Because \(X\sim N(0,1/2)\), the lower tail of \(X\) at \(-\sqrt{2L}a_\alpha+o(\sqrt L)\) is
\begin{align}
\Prob\{X< -\sqrt{2L}a_\alpha+o(\sqrt L)\}
&=\Phi\{-2a_\alpha\sqrt L+o(\sqrt L)\} \notag\\
&=\exp\{-2a_\alpha^2L+o(L)\}.
\label{eq:app-poly-rank-tail}
\end{align}
Thus
\begin{equation}\label{eq:bulk-transition}
P_{n,\lfloor n^\alpha\rfloor}=n^{-c(\alpha)+o(1)},
\qquad
c(\alpha):=2\bigl(2-\sqrt{1-\alpha}\bigr)^2.
\end{equation}
The exponent satisfies \(c(0)=2\), and for \(0<\alpha<1\),
\begin{equation}\label{eq:app-calpha-gap}
c(\alpha)>2+\alpha.
\end{equation}
The inequality \eqref{eq:app-calpha-gap} follows directly by differentiating the function \(f(\alpha)=c(\alpha)-2-\alpha\): we have \(f(0)=0\) and \(f'(\alpha)=4/\sqrt{1-\alpha}-3\ge1\) on \((0,1)\).
A block of \(O(n^\alpha)\) order indices near \(m_n\), together with all possible partners \(j>k\), contributes at most
\begin{equation}\label{eq:app-poly-block}
n^{1+\alpha}P_{n,\lfloor n^\alpha\rfloor}
=n^{1+\alpha-c(\alpha)+o(1)}=o(n^{-1+o(1)}),
\end{equation}
which is smaller than the boundary scale \(\sqrt{\log n}/n\). The extra factor \(n\), compared with a count over order indices alone, accounts for the possible second endpoint \(j\) in each candidate pair.

For order indices proportional to \(n\), the suppression is stronger. If \(k\ge \varepsilon n\), then \(Z_{(k)}\) is bounded below by a constant depending on \(\varepsilon\) with exponentially high probability, while \(Z_{(1)}=-\sqrt{2\log n}+o(\sqrt{\log n})\). Thus the one-row threshold is
\[
2Z_{(1)}-Z_{(k)}=-2\sqrt{2\log n}+O(1)+o(\sqrt{\log n}).
\]
For \(X\sim N(0,1/2)\), this gives
\[
P_{n,k}\le n^{-8+o(1)}
\]
uniformly for \(k\ge\varepsilon n\). The power \(-8\) is simply the Gaussian tail exponent associated with the leading threshold \(-2\sqrt{2\log n}\); any fixed exponent strictly larger than \(2\) would already suffice for negligibility after summing over at most \(n^2\) candidate pairs. On the exceptional event where the displayed quantile behavior fails, the order-statistic tail is exponentially small. Hence the total contribution from this range is negligible.

In contrast, \eqref{eq:Pn1-asymp} and \eqref{eq:Pnk-asymp} show that each fixed order index \(k=1,2,\ldots,K\) has one-row probability of order \(\sqrt{\log n}/n^2\). For such a fixed \(k\), there are \(n-k\) choices of the second endpoint \(j>k\), and hence
\[
(n-k)P_{n,k}=\Theta\left(\frac{\sqrt{\log n}}{n}\right).
\]
The fixed-order constants also show how much the one-row relaxation overcounts. The first order index contributes the constant \(2\sqrt{2\pi}\). For fixed \(k\ge2\), Proposition \ref{prop:app-Pnk-fixed} gives the constants
\[
\frac{24\sqrt{2\pi}}{(k+2)(k+3)},
\]
and
\[
\sum_{k=2}^{\infty}\frac{24}{(k+2)(k+3)}=6.
\]
Thus the fixed-order part of the one-row first-moment bound has constant \(8\sqrt{2\pi}\), four times the exact leading constant \(2\sqrt{2\pi}\) in Theorem \ref{thm:singleton-expansion}. This discrepancy quantifies the loss from replacing the curved two-coordinate threshold by the one-row threshold.

Thus the sum \(\sum_k(n-k)P_{n,k}\) is governed, at the level of order, by boundary order indices. Combined with \eqref{eq:k2-Pnk-bound}, this is the precise connection between \(P_{n,k}\) and \(\Prob(\kappa_n=2)\): the one-row calculation gives an upper bound at the correct order \(\sqrt{\log n}/n\), whereas the exact support-two asymptotics in the main text require the conditional product over the sharper curved events \(F_{k j}\).

\section*{Acknowledgments}
The author thanks Johannes Milz for helpful comments on the presentation.

\end{document}